\documentclass[11pt]{amsart}
\usepackage{geometry}                % See geometry.pdf to learn the layout options. There are lots.
\usepackage{amssymb}
\usepackage{amsmath}
\usepackage{xypic}
\usepackage{url}

\newtheorem{theorem}{Theorem}[section]

\newtheorem{lemma}[theorem]{Lemma}

\newtheorem{hypothesis}{Hypothesis}[theorem]

\theoremstyle{definition}
\newtheorem{definition}[theorem]{Definition}

\theoremstyle{remark}
\newtheorem{remark}[theorem]{Remark}

\newcommand{\qp}{{\mathbf{Q}}_{p}}
\newcommand{\qpbar}{\overline{\mathbf{Q}}_p}
\newcommand{\diag}{\mathrm{diag}}

\begin{document}

\title{$p$-adic Langlands functoriality for the definite unitary group}
\author{Paul-James White}
\date{\today}                                           % Activate to display a given date or no date
\email{paul-james.white@all-souls.ox.ac.uk}

\begin{abstract}
We
 formalise a notion of $p$-adic Langlands functoriality for the definite unitary group.
This extends the classical notion of Langlands functoriality to the setting of eigenvarieties.  
We apply some results of Chenevier to 
obtain some cases of $p$-adic Langlands functoriality
by interpolating known cases of classical Langlands
functoriality.
\end{abstract}

\maketitle

\section{Introduction}
The Langlands functoriality conjectures predict  deep properties of automorphic representations of  
connected reductive groups.  For definite unitary groups,  Chenevier \cite{chenevier_application_varieties_hecke}
has constructed a space of $p$-adic automorphic forms which interpolate the classical automorphic forms.
The $p$-adic automorphic forms can be parameterised by a
$p$-adic rigid analytic space known as the \emph{eigenvariety}.  It is a natural problem to extend the notion of Langlands functoriality to the setting of eigenvarieties, that is to introduce a notion of $p$-adic Langlands functoriality.
Chenevier \cite{chenevier_jacquet_langlands_p_adic} studied this problem in the case of the Jacquet-Langlands correspondence.  Chenevier was able to extend the Jacquet-Langlands correspondence to
the setting of eigenvarieties. The aim of this article is to formalise a notion of $p$-adic Langlands functoriality for the definite unitary group and to construct some examples.

The $\qpbar$-points on an eigenvariety $\mathcal{D}$ are parameterised by pairs $(\lambda, \kappa)$ 
where 
$\lambda : \mathcal{H}^- \rightarrow \qpbar^\times$ is a system of eigenvalues and $\kappa \in \mathcal{W}(\qpbar)$ a weight
coming from certain $p$-adic automorphic forms.  In order to introduce a notion of $p$-adic functoriality, we
construct under a  technical hypothesis (see Hypothesis \ref{hypothesis_1})  morphisms between the Hecke algebras and the weight spaces of definite unitary groups corresponding
to a given $L$-homomorphism.  
The constructions are natural generalisations of the  constructions appearing in classical
Langlands functoriality with one important distinction.  Unlike classical Langlands functoriality, 
the notion of $p$-adic Langlands functoriality depends upon a non-canonical choice for the refinement map in addition
to the chosen $L$-homomorphism.  Different choices for the refinement map give rise to different cases
of $p$-adic Langlands functoriality.  The need to define the refinement map comes from the fact that the
eigenvariety interpolates pairs  of automorphic forms and accessible refinements.  The need to transfer
refinements between groups requires us to introduce a non-canonical map.
Having made the appropriate definitions, we proceed to obtain certain cases of 
$p$-adic Langlands functoriality  for which
the  classical Langlands functoriality transfer is known.  The $p$-adic Langlands functorial transfer is obtained by interpolating the classical Langlands functorial transfer using work of Chenevier \cite{chenevier_jacquet_langlands_p_adic}.

Let us describe the contents of this article.  In Section \ref{section_the_eigenvariety}, we recall the construction of the 
eigenvariety for the definite unitary group.  In Section \ref{section_p_adic_funct_def}, we introduce the notion
of $p$-adic Langlands functoriality for the definite unitary group.  In Section \ref{p_adic_functoriality_examples},
we interpolate known cases of Langlands functoriality to obtain some cases of $p$-adic Langlands functoriality 
for the definite unitary group.

\subsection{Acknowledgements} 
I would like to thank Ga\"etan Chenevier,  Michael Harris, Judith Ludwig, and Alberto Minguez for helpful comments and conversations.
  I am particularly grateful to Ga\"etan Chenevier who explained how his results could be applied to the problem at hand.
I would also like to thank an anonymous referee whose comments significantly improved this article.

\subsection{Conventions and notation}
We shall normalise the Artin map from class field theory so as to send uniformizers to geometric Frobenii.
The local Langlands correspondence shall be normalised as in Harris-Taylor \cite{harris_taylor_local_langlands}.
Unless specified otherwise, we shall assume that representations are irreducible and admissible with complex coefficients.
Parabolic induction shall always refer to unitarily normalised parabolic induction.

If $k$ is a $p$-adic field, we shall fix a uniformizer  $\overline{\omega}_k$ which shall also be written as  $\overline{\omega}$.
We shall fix   field isomorphisms $\iota_p : \mathbf{C} \stackrel{\sim}{\rightarrow} \qpbar$ for all rational primes $p$.

\section{The eigenvariety}
\label{section_the_eigenvariety}
We shall recall the construction of the eigenvariety for  definite unitary groups.  The construction is  due to
Buzzard \cite{buzzard_eigenvariety} and Chenevier \cite{chenevier_application_varieties_hecke} (see also Loeffler \cite{loeffler_article}).

\subsection{The definite unitary group}
Let $E/F$ be a totally imaginary quadratic extension of a totally real field.  Let $U = U_n = U_n(E/F)$ denote
the unitary group in $n$-variables associated to the extension $E/F$ that is compact at infinity
and quasi-split at all finite places (cf. \cite[\S2]{white_endoscopy_unitary_group}).
 Such a group  exists exactly when either
 	 $n$ is odd or
	  $\frac{n\cdot [F: \mathbf{Q}]}{2}$ is even  (cf. \cite[Proposition 2.1]{white_endoscopy_unitary_group}).
 The local forms of $U$ are described below.
 \begin{itemize}
 \item For all archimedean places $\nu$ of $F$, $U_\nu = U \times_F F_\nu$ is isomorphic to the  $n$-variable real compact unitary group.
 \item For all finite places $\nu$ of $F$ that are non-split in $E$, $U_\nu \simeq U_n^*(E_\nu/F_\nu)$ the $n$-variable quasi-split unitary group associated to the extension $E_\nu/F_\nu$.
 \item For all finite places $\nu$ of $F$ that split in $E$, $U_\nu \simeq GL_n/F_\nu$.  We warn the reader that the isomorphism is non-canonical.
  It essentially depends upon the choice of a place of $E$ lying above $\nu$ (cf. \cite[\S2]{white_endoscopy_unitary_group}).  We shall fix such an isomorphism at  the finite split places
 and identify the two groups.
 \end{itemize}
 Let $\vec{n} = (n_1, \ldots, n_r) \in \mathbf{N}^r$, and let $n = \sum_{i=1}^r n_i$.
 We shall  be interested in the fibre product of the definite unitary groups
 \[
 	U_{\vec{n}} = \prod_{i=1}^r U_{n_i}
 \]
 where we have implicitly assumed that $\vec{n}$ is  such that the individual definite unitary groups $U_{n_i}$ exist.
 We shall fix an extension of $U_{\vec{n}}/F$ to a smooth group scheme $U_{\vec{n}}/\mathcal{O}_{F, S_{\mathrm{ram}}}$
 where $S_{\mathrm{ram}}$ is the set of archimedean places of $F$ and the finite places of $F$ that ramify in $E$.
 
 \subsection{Some subgroups}
 Let $\nu$ be a finite place of $F$.  We shall  recall some important subgroups of $GL_m(F_\nu)$ below.
 \begin{itemize}
 \item Let $T_\nu \subset GL_m(F_\nu)$ be the maximal split torus consisting of the  diagonal matrices.
 \item Let $T^0_\nu \subset T_\nu$ be the maximal compact subgroup consisting of the diagonal matrices in $GL_m(\mathcal{O}_{F_\nu})$.
 \item Let $T_\nu^- \subset T_\nu$ (resp. $T_\nu^{--} \subset T_\nu$)  be the submonoid consisting of the elements of the form
 	\[
	\mathrm{diag}(x_1,\ldots, x_m)
	\]
	where $\mathbf{v}(x_1) \geq \mathbf{v}(x_2) \geq \cdots \geq \mathbf{v}(x_m)$
	(resp. $\mathbf{v}(x_1) > \cdots  > \mathbf{v}(x_m)$) where $\mathbf{v}(x)$ denotes the valuation of an element $x \in F_\nu$.
	\item Let $T_\nu^{\overline{\omega}} \subset T_\nu$ (resp. $T_\nu^{\overline{\omega}, -} \subset T_\nu$)
	be the subgroup (resp. submonoid) consisting of the elements of the form
	 	\[
		\mathrm{diag}(\overline{\omega}^{\lambda_1},\ldots, \overline{\omega}^{\lambda_m})
	\]
	where $\lambda_1, \ldots, \lambda_m \in \mathbf{Z}$ (resp.
	$\lambda_1 \geq \cdots \geq \lambda_m$).

\item Let $B_\nu \subset GL_m(F_\nu)$ be the Borel subgroup consisting of the upper triangular matrices.
\item Let $N_\nu \subset GL_m(F_\nu)$ be the unipotent subgroup consisting of upper triangular matrices with $1$ on their diagonals.
\item Let $I_\nu \subset GL_m(\mathcal{O}_{\nu})$ be the Iwahori subgroup consisting of the upper triangular matrices  modulo $\overline{\omega}$.
\item Let $M_\nu \subset GL_m(F_\nu)$ be the submonoid generated by $I_\nu$ and $T_\nu^-$.
 \end{itemize}
 We have canonical isomorphisms $T_\nu/T_\nu^0 \simeq T_\nu^{\overline{\omega}}$ and
 $T_\nu^-/T_\nu^0 \simeq T_\nu^{\overline{\omega}, -}$.
 If $\nu$ splits in $E$, then we  define the analogous subgroups of
 ${U_{\vec{n}}(F_\nu) = GL_{n_1}(F_\nu) \times \cdots \times GL_{n_r}(F_\nu)}$
 in the obvious way (e.g. $T_\nu = T_{GL_{n_1},\nu} \times \cdots \times T_{GL_{n_r},\nu}$ where   $T_{GL_{n_i},\nu}$ denotes
  the 
 previously  defined subgroup of $GL_{n_i}(F_\nu)$  of diagonal matrices).  If $S$ is a finite set of non-archimedean places of $F$ that split in $E$, then we   define the analogous subgroups of $\prod_{\nu \in S} U_{\vec{n}}(F_\nu)$  in the obvious way (e.g. $T_S = \prod_{\nu \in S} T_\nu$).
 
 \subsection{The datum}
 An eigenvariety for $U_{\vec{n}}$ depends upon the choice of a datum $(p, S, e, \phi)$ whose elements are  as follows.
 \begin{itemize}
 	\item $p$ is a rational prime that  splits completely in $E$.
	\item $S_p$ is the set of places of $F$ lying above $p$.
	\item $S_\infty$ is the set of archimedean places of $F$.
	\item $S$ is a finite set of non-archimedean places of $F$ containing all the
	 places of $F$ that are ramified in $E$ and such that $S \cap S_p = \emptyset$.
	\item $K^S = \prod_{\nu\not\in S} K_\nu \subset U_{\vec{n}}(\mathbf{A}^S_{f})$ is the compact open
	subgroup such that $K_\nu = U_{\vec{n}}(\mathcal{O}_\nu)$ is maximal hyperspecial for all $\nu \not\in S_p \cup S$ and $K_\nu = I_\nu$ for all $\nu \in S_p$.
	\item For all $\nu \in S$, $e_\nu$ is a non-trivial idempotent of the Hecke algebra
	$\mathcal{C}_c^\infty\left(U_{\vec{n}}(F_\nu), \qpbar \right)$.
	\item $e = \otimes_{\nu \in S} e_\nu \otimes \mathbf{1}_{K^{S}}$  seen as an idempotent of the 
	Hecke algebra 	$\mathcal{C}_c^\infty(U_{\vec{n}}(\mathbf{A}_{f}), \qpbar)$ where $\mathbf{1}_{K^{S}}$ denotes the identity function on the compact open subgroup
	$K^{S}$.
	\item $\phi \in \qpbar[T_{S_p}^{--}]$.
 \end{itemize}
 
  We shall 
  identify the sets $S_\infty$ and $S_p$ via 
   the field isomorphism $\iota_p : \mathbf{C} \stackrel{\sim}{\rightarrow} \qpbar$.
   Explicitly this bijection is given as follows.
 \begin{eqnarray*}
 \iota_p^{-1} : \mathrm{Hom}(F, \qpbar) &\rightarrow& \mathrm{Hom}(F, \mathbf{C})	\\
 \nu &\mapsto& \iota_p^{-1} \circ \nu
 \end{eqnarray*}
 
 \subsection{The Hecke algebras}
 We shall recall the commutative Hecke algebras that are relevant to a chosen datum $(p, S, e, \phi)$.
 The \emph{spherical Hecke algebra} of $U_{\vec{n}}$ outside of $S \cup S_p$ is defined to be
 \[
 	\mathcal{H}_{\mathrm{ur}} = \mathcal{C}_{\mathrm{c}}^\infty(
		K^{S\cup S_p} \backslash U_{\vec{n}}(\mathbf{A}_f^{S \cup S_p}) / K^{S \cup S_p}, \qpbar).
 \]
 If $\nu \in S_p$, then we shall have need of the   \emph{Atkin-Lehner} sub-algebra $\mathfrak{U}_\nu^-$
of the Hecke-Iwahori algebra
\[
 	 \mathcal{A} = \mathcal{C}_{\mathrm{c}}^\infty(
		I_\nu \backslash U_{\vec{n}}(F_\nu) / I_\nu, \qpbar).
\]
The $\qpbar$-subalgebra $\mathfrak{U}_\nu^- \subset \mathcal{A}$ is defined to be the sub-algebra generated by the identity functions
$\mathbf{1}_{I_\nu t I_\nu}$ for all $t \in T_\nu^-$.  
We finish by defining the commutative  $\qpbar$-algebras
\[	\mathfrak{U}^- = \otimes_{\nu\in S_p} \mathfrak{U}_\nu^- \ \text{ and } \
	\mathcal{H}^- = \mathfrak{U}^- \otimes \mathcal{H}_{\mathrm{ur}}.
\]

 \subsection{The weight space}
 \label{subsection_the_weight_space}
 If $\nu \in S_p$, then one defines $\mathcal{W}_\nu$  to be the $\qp$-rigid analytic space
 that represents the functor 
 \[
 	\mathcal{W}_\nu = \mathrm{Hom}_{\mathrm{cts-gp}}(T_\nu^0, \mathbf{G}_m^{\mathrm{rig}}).
 \]
 The \emph{weight space} is defined to be the $\qp$-rigid analytic space
 \[
 	\mathcal{W} = \prod_{\nu\in S_p} \mathcal{W}_\nu.
 \]
 \begin{remark}
 The weight space $\mathcal{W}$ is isomorphic to the finite disjoint union of open balls of dimension $n [F : \mathbf{Q}]$.
 \end{remark}

    Let $\omega$ be a real place of $F$ and let $\nu = \iota_p \circ \omega$ be the corresponding place above $p$.  
 We have the group embedding
  \[
 	U_{\vec{n}}(F_\omega) \hookrightarrow U_{\vec{n}}(E_\omega)
		= (GL_{n_1} \times \cdots \times GL_{n_r} )(E_\omega).
 \]
 We shall fix   $E \hookrightarrow \mathbf{C}$ a field embedding above $\omega$. This induces an isomorphism
 $E_\omega \simeq \mathbf{C}$ from which we obtain the group embedding
 \[
 	U_{\vec{n}}(F_\omega) \hookrightarrow (GL_{n_1} \times \cdots \times GL_{n_r})(\mathbf{C})
	\stackrel{\iota_p}{\rightarrow} (GL_{n_1} \times \cdots \times GL_{n_r})(\overline{F}_\nu)
	= U_{\vec{n}}(\overline{F}_\nu).
 \]
 The group $U_{\vec{n}}\left(F_\omega\right)$ is compact. Its irreducible admissible representations are obtained by restriction from the finite dimensional irreducible algebraic representations of 
 $U_{\vec{n}}(\overline{F}_\nu)$ (cf. \cite[\S6.7]{bellaiche_chenevier_book}).
  Such representations  are classified by their \emph{highest weight characters}
 (with respect to the Borel $B_\nu$).  
  If $\pi_\omega$ is an irreducible admissible representation of $U_{\vec{n}}(F_\omega)$, then
   we shall write
 \[
 	\kappa(\pi_\omega) : T_\nu \rightarrow \qpbar^\times
 \]
 for the highest weight character of $\pi_\omega$.  The character is of the following form. 
 \begin{eqnarray*}
 \kappa(\pi_\omega): 	T_{1,\nu} \times \cdots \times T_{r,\nu} &\rightarrow& \qpbar^\times	\\
	\diag\left(x_{1,1}, \ldots, x_{1,n_1}\right) \times \cdots \times \diag\left(x_{r,1}, \ldots, x_{r,n_r}\right) 
		&\mapsto& \prod_{i=1}^r \prod_{j=1}^{n_i} x_{i,j}^{k_{i,j}}
 \end{eqnarray*}
 where 
  $k_{i,j} \in \mathbf{Z}$ and $k_{i,j} \geq k_{i,k}$ 
  for all $1 \leq i \leq r$ and
  for all $1 \leq j \leq k \leq n_i$.
  The tuple of integers 
  \[
  \{k_{i,j} : 1 \leq i \leq r, 1 \leq j \leq n_i\}
  \]
   is called the set of  \emph{highest weights} of
 the representation
 $\pi_\omega$.
 The highest weight is said to be \emph{regular} if
 \[
 	k_{i,j} > k_{i,k}
 \]
 for all $1 \leq i \leq r$ and
  for all $1 \leq j < k \leq n_i$.  
  
      If $\pi_\infty$ is an irreducible admissible representation of $U_{\vec{n}}(\mathbf{A}_{\infty})$, then
     $\pi_\infty$ is said have \emph{regular highest weight} if
  for all $\omega|\infty$,  the highest weight of $\pi_\omega$ is regular.
  We shall write
 \[
 	\kappa(\pi_\infty) = \otimes_{\omega|\infty} \kappa(\pi_\omega) : T_{S_p} \rightarrow \qpbar^\times
 \]
  for the \emph{highest weight character} of $\pi_\infty$. 
  Such a character (or more precisely its restriction to $T^0_{S_p}$) gives a $\qpbar$-valued point in the weight space
 \[
 	\kappa(\pi_\infty) \in \mathcal{W}(\qpbar).
 \]
 A point $\delta \in \mathcal{W}(\qpbar)$ is said to be $\emph{regular classical}$ if it is of the form
 $\kappa(\pi_\infty)$ for some $\pi_\infty$ of regular highest weight.

 \subsection{Refinements}
 We shall recall the notion of an accessible refinement following \cite[\S1.4]{chenevier_application_varieties_hecke}
 (see also \cite[\S6.4]{bellaiche_chenevier_book}).
 Let $\nu \in S_p$ and let $\pi_\nu$ be an irreducible admissible representation of 
 $U_{\vec{n}}(F_\nu) = GL_{n_1}(F_\nu) \times \cdots \times GL_{n_r}(F_\nu)$.
 A \emph{refinement} of $\pi_\nu$ is an unramified character
 \[
 	\chi_\nu : T_\nu^{\overline{\omega}} \simeq T_\nu/T^0_\nu \rightarrow \qpbar^\times
 \]
 such that $\pi_\nu$ appears as a constituent of the  induced
 representation $\mathrm{Ind}_{B_\nu}^{U_{\vec{n}, \nu}} \iota_p^{-1} \chi_\nu$.
 A refinement is said to be \emph{accessible} if $\pi_\nu$ appears as a \emph{subrepresentation} of the
  induced representation.  
  
  \begin{remark}
 A continuous group homomorphism $\chi_\nu :  T_\nu^{\overline{\omega}, -}  \rightarrow \qpbar^\times$ uniquely extends to
a continuous group homomorphism $\chi_\nu : T_\nu^{\overline{\omega}} \rightarrow \qpbar^\times$.
These two equivalent descriptions  shall be used interchangeably throughout this article.
 \end{remark}
  
  The following result shows that accessible refinements exist exactly for 
  the
 representations with Iwahori-invariant vectors.
 \begin{lemma}[Bernstein, Borel, Casselman, Matsumoto]
 \label{lemma_bernstein_borel_casselman_matsumoto}
 Let $\nu \in S_p$.  
\begin{itemize}
	\item We have the equality $M_\nu = \sqcup_{t \in T_\nu^{\overline{\omega},-}} I_\nu t I_\nu$.  Furthermore, the map
	\[
		\tau : M_\nu \rightarrow  T_\nu^{\overline{\omega},-}
	\]
	defined by the relation $m \in I_\nu \tau(m) I_\nu$ is a surjective multiplicative homomorphism.
\item
	The map
	\begin{eqnarray*}
		 T_\nu^{\overline{\omega},-} &\rightarrow& \mathfrak{U}_\nu^-	\\
		t &\mapsto&  \mathbf{1}_{I_\nu t I_\nu}
	\end{eqnarray*}
	is multiplicative, and it extends to give the $\qpbar$-algebra isomorphism
	\[
		\qpbar[ T_\nu^{\overline{\omega},-}] \stackrel{\sim}{\rightarrow} \mathfrak{U}_\nu^-.
	\]
\item Let $\pi_\nu$ be an irreducible admissible representation of $U_{\vec{n}}\left(F_\nu\right)$.
  We shall view $\pi_\nu$ as a 
$\qpbar[ T_\nu^{\overline{\omega},-}]$-module via the previous isomorphism.  Then
\[
	\left(\pi_\nu^{I_\nu}\right)^{\mathbf{ss}} \simeq \bigoplus  \iota_p^{-1}\circ (\chi_\nu  \delta_{B_\nu}^{-{1}/{2}})
\]
where $\chi_\nu$ runs through the accessible refinements of $\pi_\nu$, $\mathbf{ss}$ denotes  semi-simplification,
and $\delta_{B_\nu}: T_\nu/T^0_\nu \rightarrow \qpbar^\times$ denotes the modulus character viewed with coefficients in $\qpbar$ via the isomorphism $\iota_p$.
\end{itemize}
 \end{lemma}
 \begin{proof}
  \cite[\S6.4]{bellaiche_chenevier_book}
 \end{proof}
  
 \begin{remark}
To give a refinement of  an irreducible admissible representation of $U_{\vec{n}}(F_\nu)$
with an Iwahori invariant vector is equivalent to giving
  an ordering of the eigenvalues 
 of the semi-simple conjugacy class of $(GL_{n_1} \times \cdots \times GL_{n_r})(\qpbar)$ associated to the geometric Frobenius  via the local Langlands correspondence.  In general, not all refinements will be accessible.  For example, the Steinberg representation has a single accessible refinement.
 \end{remark}

 If $\pi$ is an automorphic representation of $U_{\vec{n}}(\mathbf{A}_F)$,  
 then a \emph{refinement} (resp. \emph{accessible refinement}) of $\pi$ is an unramified character
 \[
 	\chi = \otimes_{\nu \in S_p} \chi_\nu: T_{S_p}^{\overline{\omega}} \simeq  T_{S_p}/T_{S_p}^0 \rightarrow \qpbar^\times
 \]
 such that  $\chi_\nu$ is a  refinement (resp. accessible refinement) of the representation
 $\pi_\nu$ for all $\nu \in S_p$.
 To avoid  problems of algebraicity in the construction of the eigenvariety,
 one normalises a refinement $\chi$ of $\pi$ as follows
 \[
 	\nu(\pi, \chi) = \kappa(\pi_\infty)  \chi  \delta_{B_{S_p}}^{-1/2} : T_{S_p}^{\overline{\omega}} \rightarrow \qpbar^\times
 \]
  where $\kappa(\pi_\infty)$ is viewed by restriction as a character of $T_{S_p}^{\overline{\omega}}$.
  
  \subsection{Module valued automorphic forms}
  We shall recall the notion of a module valued automorphic form (see \cite{gross_algebraic_mf} for the general theory).
  Following  the discussion in Section \ref{subsection_the_weight_space}, we have the  commutative diagrams
  \[
 \xymatrix{
  	F \ar@{^{(}->}[d] \ar@{^{(}->}[r] & \mathbf{A}_{S_p} \ar@{^{(}->}[d]	& & U_{\vec{n}}(F) \ar@{^{(}->}[d]  \ar@{^{(}->}[r] & U_{\vec{n}}(\mathbf{A}_{S_p}) \ar@{^{(}->}[d]	\\
	\mathbf{A}_\infty \ar@{^{(}->}[r] & \prod_{\omega|\infty} \mathbf{C} & & U_{\vec{n}}(\mathbf{A}_\infty) \ar@{^{(}->}[r] & \prod_{\omega|\infty} GL_{n_1} \times \cdots GL_{n_r}(\mathbf{C})
  }
  \]
  where $F$ is embedded diagonally into both $\mathbf{A}_\infty$ and $\mathbf{A}_{S_p}$.
  An irreducible admissible representation $W$ of $U_{\vec{n}}(\mathbf{A}_\infty)$ is obtained via restriction
  from a unique irreducible algebraic representation of $\prod_{\omega|\infty} GL_{n_1} \times \cdots GL_{n_r}(\mathbf{C})$, which shall also be denoted by $W$.   
  Using the ring isomorphism $\iota_p : \mathbf{C} \stackrel{\sim}{\rightarrow} \qpbar$, we shall view $W$ 
  as a representation with coefficients in $\qpbar$.   The above diagram equips $W$ with a
  group action of $U_{\vec{n}}(\mathbf{A}_{S_p})$.
  We  define the $\qpbar$-vector space
  \[
  \mathcal{A}(U_{\vec{n}}, W)  =
  \begin{cases}
		f : U_{\vec{n}}(F) \backslash U_{\vec{n}}(\mathbf{A}_f) \rightarrow W :
		\text{$f$ is smooth outside of $S_p$, } \\
			\ \ \ \ \ \ \ \ \	 f(g  k_{S_p}) = k_{S_p}^{-1} f(g),  \forall g \in U_{\vec{n}}(\mathbf{A}_f), \forall k_{S_p} \in I_{S_p}
  \end{cases}
  \]
  where by $f$ is \emph{smooth outside} of $S_p$, we mean that $f$ is invariant under right translation by some compact open subgroup of $U_{\vec{n}}(\mathbf{A}^{S_p}_{f})$.
The monoid $U_{\vec{n}}(\mathbf{A}^{S_p}_f) \times M_{S_p}$ acts on the space 
as follows
\[
(\gamma k_{S_p} f)(g) = k_{S_p} f(g \gamma k_{S_p}) \ \ \ \forall \gamma \in U_{\vec{n}}(\mathbf{A}_f^{S_p})
\ \ \forall  k \in M_{S_p} \ \   \forall f \in \mathcal{A}(U_{\vec{n}}, W).
\]
This action allows us to view  $\mathcal{A}\left(U_{\vec{n}}, W\right)$  as a
$\mathcal{H}^-$-module.

The relationship between these modules and the usual notion of an automorphic representation
is given by the following result.
\begin{lemma}
\label{lemma_description_of_module_space_classical_terms_automorphic_representations}
There exists an isomorphism of $\mathcal{H}^-$-modules
\[
\mathcal{A}(U_{\vec{n}}, W) \bigotimes_{\qpbar, \iota^{-1}_p} \mathbf{C} \simeq
\bigoplus_{\Pi} m(\Pi) \Pi_f^{I_{S_p}}
\]
where $\Pi$ runs through the automorphic representations of $U_{\vec{n}}$
 such that 
$\Pi_\infty \simeq W$ and $\Pi_{S_p}^{I_{S_p}} \neq 0$, and 
 $m(\Pi)$ denotes the multiplicity of $\Pi$ in the discrete automorphic spectrum of $U_{\vec{n}}$.
\end{lemma}
\begin{proof}
\cite[Chapter II]{gross_algebraic_mf}
\end{proof}

 \subsection{Systems of Banach modules}
 We shall recall Buzzard's \cite{buzzard_eigenvariety} property (PR) for Banach modules and the notion of a system of Banach modules following Chenevier \cite[\S1]{chenevier_jacquet_langlands_p_adic}.  Let $A$ be a commutative noetherian $\qpbar$-Banach algebra.
 An $A$-Banach module ${M}$ is said to satisfy the \emph{property (PR)} if there exists another $A$-Banach module 
 ${M}'$ such that ${M} \oplus {M}'$ is isomorphic (but not necessarily isometrically isomorphic) to 
 an orthonormal $A$-Banach module.
 A \emph{system of (PR) $A$-Banach modules} is defined to be a set
 \[
 	\mathbf{M} = \left\{ {M}_i ; \iota_i : i \in \mathbf{N} 
	\right\}
 \]
 where for all $i \in \mathbf{N}$, 
 \begin{itemize}
 	\item $M_i$ is an $A$-Banach module satisfying the property (PR), and
	\item $\iota_i : M_i \rightarrow M_{i+1}$ is an $A$-linear compact morphism.
 \end{itemize}
 The corresponding inverse limit  is written as 
$\mathbf{M}^{\dagger} = \varprojlim M_i$.
Let $\mathcal{S}$ be a reduced separated $\qpbar$-rigid analytic space.
A \emph{sheaf of (PR) Banach modules} on $\mathcal{S}$ is a sheaf of modules $B$ on $\mathcal{S}$ such that
\begin{itemize}
	\item for all open affinoids  $V \subset \mathcal{S}$, $B(V)$ is an $\mathcal{O}(V)$-Banach module satisfying the property (PR), and
	\item for all open affinoids $V' \subset V \subset \mathcal{S}$, the base change morphism 
	\[
		B(V) \widehat{\otimes}_{\mathcal{O}(V)} \mathcal{O}(V')
			\rightarrow B(V')
	\]
	is an isomorphism of $\mathcal{O}(V')$-Banach modules.
\end{itemize}

A \emph{system of (PR) Banach modules} on $\mathcal{S}$ is a set
\[
	\mathbf{M} = \left\{
		M_i; \iota_i : i \in \mathbf{N}
	\right\}
\]
where for all $i \in \mathbf{N}$,
\begin{itemize}
	\item $M_i$ is a sheaf of (PR) Banach modules on $\mathcal{S}$, and
	\item $\iota_i : M_i \rightarrow M_{i+1}$ is a sheaf morphism such that
	for all open affinoids $V \subset \mathcal{S}$, 
	$\iota_i(V) : M_i(V) \rightarrow M_{i+1}(V)$ is an 
	$\mathcal{O}(V)$-linear compact morphism.
\end{itemize}
 If $V \subset \mathcal{S}$ is an open affinoid, then we shall denote the corresponding system of 
 (PR) $\mathcal{O}\left(V\right)$-Banach modules by
 \[
 	\mathbf{M}(V) = \left\{
		{M}_i(V); \iota_i : i \in \mathbf{N}
	\right\}.
 \]
If $x \in \mathcal{S}(\qpbar)$,
then we shall denote the corresponding system of (PR) $\qpbar$-Banach modules by 
\[
\mathbf{M}_x = \left\{ {M}_{i,x}; \iota_i : i \in \mathbf{N} \right\}
\]
 where  for $i \in \mathbf{N}$, 
${M}_{i, x} = {M}_i\left(V\right) \widehat{\otimes}_{\mathcal{O}\left(V\right)} \qpbar$  
 where   $V \subset \mathcal{S}$ is an  open affinoid neighbourhood of $x$ and the map
 $\mathcal{O}(V) \rightarrow \qpbar$ is the one given by the point $x$.
  An \emph{endomorphism} of $\mathbf{M}$ is a set
 \[
 \phi = \left\{ \phi_i(V)  :  V \subset \mathcal{S} \text{  open affinoid,  $i \in \mathbf{N}$  sufficiently large} \right\}
 \]
where each $\phi_i(V) : {M}_i(V) \rightarrow {M}_i(V)$ is a continuous 
$\mathcal{O}(V)$-linear endomorphism, and the $\phi_i(V)$ commute with both the $\iota_i$ and the 
base change morphisms between  open affinoids $V \subset V' \subset \mathcal{S}$.
%Two endomorphisms $\phi$ and $\phi'$ are said to be \emph{equivalent} if 
%for all open affinoids $V \subset \mathcal{S}$ and for all $i$ sufficiently large, $\phi_i(V) = \phi'_i(V)$.

 Let $\mathrm{Comp}(\mathbf{M}) \subset \mathrm{End}(\mathbf{M})$
 be the two-sided ideal consisting of  endomorphisms $\phi$ such that for all open affinoids $V \subset \mathcal{S}$,
 there exists for $i$ sufficiently large, a continuous $\mathcal{O}(V)$-linear morphism
 $\psi_i(V) : M_{i+1}(V) \rightarrow M_i(V)$ such that the following diagram commutes.
 \[
 \xymatrix{
 	M_i(V) \ar[d]^{\iota_i(V)} \ar[r]^{\phi_i(V)} & M_i(V) \ar[d]^{\iota_i(V)} \\
	M_{i+1}(V) \ar[ru]^{\psi_i(V)} \ar[r]^{\phi_{i+1}(V)} & M_{i+1}(V)
 }
 \]

 \subsection{Buzzard's eigenvariety machine}
 We shall consider a datum 
 $\left(\mathcal{S}, \mathbf{M}, T, \phi\right)$ where
 \begin{itemize}
 	\item $\mathcal{S}$ is a reduced separated $\qpbar$-rigid analytic space,
	\item $\mathbf{M}$ is a system of (PR) Banach modules on $\mathcal{S}$,
	\item $T$ is a commutative $\qpbar$-algebra equipped with a homomorphism
	$T \rightarrow \mathrm{End}(\mathbf{M})$, and
	\item $\phi \in T$ acts compactly on $\mathbf{M}$ (that is $\phi$ acts compactly on the ${M}_i(V)$ for all $i \in \mathbf{N}$ and for all open affinoids $V \subset \mathcal{S}$) and satisfies the compatibility condition
	$\phi \in \mathrm{Comp}(\mathbf{M})$.
 \end{itemize}
A $\qpbar$-valued \emph{system of eigenvalues} for
 $\left(\mathcal{S}, \mathbf{M}, T, \phi\right)$ is a pair $\left(\lambda, x\right)$ where
 \begin{itemize}
 	\item $x \in \mathcal{S}(\qpbar)$ and
	\item $\lambda : T \rightarrow \qpbar$ is a $\qpbar$-algebra homomorphism such that
	there exists a non-zero $m \in \mathbf{M}_x^\dagger$ for which
	$\alpha(m) = \lambda(\alpha) \cdot m$ for all $\alpha \in T$.
 \end{itemize}
 A system of eigenvalues $\left(\lambda, x\right)$ is said to be \emph{$\phi$-finite} if $\lambda(\phi) \neq 0$.
 
 \begin{theorem}
 There exists a tuple $\left(\mathcal{D}, \psi, \kappa \right)$ where
 \begin{itemize}
 	\item $\mathcal{D}$ is a $\qpbar$-rigid analytic space,
	\item $\psi : T \rightarrow \mathcal{O}(\mathcal{D})$ is a $\qpbar$-algebra homomorphism, and
	\item $\kappa : \mathcal{D} \rightarrow \mathcal{S}$ is a morphism of rigid analytic spaces
 \end{itemize}
 such that
 \begin{itemize}
 	\item the map $\nu = (\kappa, \psi(\phi)^{-1}) : \mathcal{D} \rightarrow \mathcal{S} \times \mathbf{G}^{\mathrm{rig}}_m$ is a finite morphism,
	\item for all open affinoids $V \subset \mathcal{S} \times \mathbf{G}^{\mathrm{rig}}_m$, the natural map
	\[
		\psi \widehat{\otimes} \nu^* : T \widehat{\otimes} \mathcal{O}(V) \rightarrow
			\mathcal{O}(\nu^{-1}(V))
	\]
	is surjective, and
	\item the natural evaluation map
	\[
		\mathcal{D}(\qpbar) \rightarrow \mathrm{Hom}(T, \qpbar), \ 
		x \mapsto \psi_x : (h \mapsto \psi(h)(x))
	\]
	induces a bijection $x \mapsto (\psi_x, \kappa(x))$ between the set of $\qpbar$-valued points
	of $\mathcal{D}$ and the set of $\phi$-finite $\qpbar$-valued systems of eigenvalues for
	 $\left(\mathcal{S}, \mathbf{M}, T, \phi\right)$.
 \end{itemize}
 Furthermore if $\mathcal{D}$ is reduced, then $\left(\mathcal{D}, \psi, \kappa\right)$ is uniquely determined by these properties. 
 \end{theorem}
 \begin{proof}
 The eigenvariety was constructed by Buzzard \cite[\S5]{buzzard_eigenvariety} generalising an earlier construction
 of Coleman-Mazur \cite{coleman_mazur}.  The uniqueness, a formal property, was obtained by Bella\"iche-Chenevier \cite[Proposition 7.2.8]{bellaiche_chenevier_book}.
 \end{proof}
 
 \subsection{$p$-adic forms of type $\left(p, S, e, \phi\right)$}
 If $W$ is an irreducible admissible representation of $U_{\vec{n}}(\mathbf{A}_\infty)$, then
we define the $\mathcal{H}^-$-module
\[
	S^{\mathrm{cl}}_{\kappa\left(W\right)} = e \mathcal{A}\left(U_{\vec{n}}, W\right).
\]
Chenevier \cite[\S2]{chenevier_application_varieties_hecke} has constructed a system of (PR) Banach modules on the weight space $\mathcal{W}$ that interpolate the spaces $S^{\mathrm{cl}}_{\kappa\left(W\right)}$.  
The \emph{system of $p$-adic forms} of type $\left(p, S,e, \phi\right)$, denoted $\mathbf{S} = \left\{S_i; \iota_i : i \in \mathbf{N}\right\}$, satisfies the following properties.
\begin{itemize}
\item There exists a $\qpbar$-algebra homomorphism $\mathcal{H}^- \rightarrow \mathrm{End}(\mathbf{S})$.
\item If $t \in \qpbar[T^{--}_{S_p}]$, then $t$ viewed as an element of $\mathcal{H}^-$ acts compactly on $\mathbf{S}$ and satisfies the compatibility condition $t \in \mathrm{Comp}(\mathbf{S})$.
\item For all $W$, there exists a natural embedding of $\mathcal{H}^-$-modules,
	\[
		{S}_{\kappa(W)}^{\mathrm{cl}} \otimes \kappa(W) \hookrightarrow \mathbf{S}^\dagger_{\kappa(W)}
	\]
	where $\kappa(W): T_{S_p} \rightarrow \qpbar^\times$ is viewed by restriction as a character of $T_{S_p}^{\overline{\omega}, -}$.
	Chenevier \cite[Proposition 2.17]{chenevier_application_varieties_hecke} also obtains a \emph{small slope is classical} type result.
	 This
	allows us to deduce that certain $p$-adic forms $f \in \mathbf{S}^\dagger_{\kappa(W)}$ are classical,
	that is $f$ lies in the image of the above embedding.
\end{itemize}
 
 \subsection{The eigenvariety of type $\left(p, S, e, \phi\right)$}
Let $\pi$ be an automorphic representation of $U_{\vec{n}}(\mathbf{A}_F)$ such that $e(\pi_f) \neq 0$.
The complex vector space $(\pi_f^{S \cup S_p})^{K^{S \cup S_p}}$ is  $1$-dimensional and $\mathcal{H}_{\mathrm{ur}}$
acts on this space via scalar multiplication.  We shall write
\[
	\psi_{\mathrm{ur}}(\pi) : \mathcal{H}_{\mathrm{ur}} \rightarrow \qpbar
\]
for the corresponding  homomorphism composed with $\iota_p$.  
Let 
\[
\mathcal{Z} \subset \mathrm{Hom}(\mathcal{H}^-, \qpbar) \times \mathcal{W}(\qpbar)
\]
denote the subset of pairs $(\nu(\pi, \chi) \otimes \psi_{\mathrm{ur}}(\pi), \kappa(\pi_\infty))$ where $\pi$ runs through the  automorphic representations of the above form and $\chi$ the  accessible refinements of $\pi$.
 
\begin{theorem}
\label{theorem_existence_eigenvariety}
There exists a unique  tuple  $\left(\mathcal{D}, \psi, \kappa, Z\right)$ where
\begin{itemize}
	\item $\mathcal{D}$ is a reduced rigid analytic space over $\qpbar$,
	\item $\psi : \mathcal{H}^- \rightarrow \mathcal{O}(D)$ is a $\qpbar$-algebra homomorphism,
	\item $\kappa : \mathcal{D} \rightarrow \mathcal{W}$ is a morphism of rigid analytic spaces, and
	\item $Z \subset \mathcal{D}(\qpbar)$ is an accumulation and Zariski-dense subset
	(cf. \cite[\S3.3.1]{bellaiche_chenevier_book})
\end{itemize}
 such that
 \begin{itemize}
 	\item the map $\nu = (\kappa, \psi(\phi)^{-1}) : \mathcal{D} \rightarrow \mathcal{W} \times \mathbf{G}^{\mathrm{rig}}_m$ is a finite morphism,
	\item for all open affinoids $V \subset \mathcal{W} \times \mathbf{G}^{\mathrm{rig}}_m$, the natural map
	\[
		\psi \widehat{\otimes} \nu^* : \mathcal{H}^- \widehat{\otimes} \mathcal{O}(V) \rightarrow
			\mathcal{O}(\nu^{-1}(V))
	\]
	is surjective, and
	\item the natural evaluation map 
	$\mathcal{D}(\qpbar) \rightarrow \mathrm{Hom}_{\mathrm{ring}}(\mathcal{H}^-, \qpbar)$
	\[
		x \mapsto \psi_x : (h \mapsto \psi(h)(x))
	\]
	induces a bijection $Z \stackrel{\sim}{\rightarrow} \mathcal{Z}$, $z\mapsto \left(\psi_z, \kappa(z)\right)$.
\end{itemize}
In addition, it also satisfies the following properties.
\begin{itemize}
	\item $\mathcal{D}$ is equidimensional of dimension $\dim(\mathcal{W}) = n [F : \mathbf{Q}]$.
	More precisely, 
	 $\mathcal{D}$ has a canonical admissible covering which is given by the open affinoids $\Omega \subset \mathcal{D}$ such that $\kappa(\Omega)$ is an open affinoid and the morphism
	 $\kappa|_{\Omega} : \Omega \rightarrow \kappa(\Omega)$ is finite and surjective when restricted to each irreducible component of $\Omega$.  Furthermore, the image by $\kappa$ of each irreducible component of $\mathcal{D}$ is Zariski-open in $\mathcal{W}$.
	 \item $\psi(\mathcal{H}_{\mathrm{ur}}) \subset  \mathcal{O}(\mathcal{D})^{\leq 1}$ where
	 	$\mathcal{O}(\mathcal{D})^{\leq 1} \subset \mathcal{O}(\mathcal{D})$ denotes the subring
		of  functions bounded by $1$.
\end{itemize}

\end{theorem}
\begin{proof}
The eigenvariety is constructed
by applying the eigenvariety machine to the $p$-adic forms of type $\left(p, S, e, \phi\right)$ (cf. \cite[Theorem 1.6]{chenevier_application_varieties_hecke}).  
We remark that the density of classical points follows from a small slope is classical type result
(cf. \cite[Proposition 2.17]{chenevier_application_varieties_hecke}) whilst
the fact that $\mathcal{D}$ is reduced follows from a result of Chenevier \cite[Proposition 3.9]{chenevier_jacquet_langlands_p_adic}. 
\end{proof}

 \section{$p$-adic Langlands functoriality: definitions}
 \label{section_p_adic_funct_def}
We shall generalise the notion of Langlands functoriality to the setting of eigenvarieties.
The setup is as follows.
Let $H=U_{\vec{n}}$, and $G=U_{{m}}$ where $\vec{n} = \left(n_1,\ldots, n_r\right)$ and $m \in \mathbf{N}$.
Let $\left(p, S, e_H, \phi_H\right)$ and $\left(p, S, e_G, \phi_G\right)$ be data.  (Concerning notation, we shall add a subscript  $H$ or $G$ to previously defined   objects  to indicate the group to which they are associated.)  Let
\[
	\xi : {}^L H \rightarrow {}^L G
\]
be an $L$-homomorphism.  Recall that ${}^L H = \widehat{H} \rtimes W_F$ where
$\widehat{H} = GL_{n_1} \times\cdots \times GL_{n_r}(\mathbf{C})$ and the Weil group $W_F$ acts via projection onto
$\mathrm{Gal}(E/F) = \left\{1,c\right\}$ where $c$ acts via the isomorphism
\begin{eqnarray*}
c : GL_{n_1} \times \cdots \times GL_{n_r}(\mathbf{C}) &\rightarrow& GL_{n_1}\times \cdots \times GL_{n_r}(\mathbf{C})  \\
 g_1 \times \cdots \times g_r &\mapsto& \Phi_{n_1} {}^t g_1^{-1} \Phi_{n_1}^{-1} \times \cdots \times \Phi_{n_r} {}^t g_r^{-1} \Phi_{n_r}^{-1}
\end{eqnarray*}
where
\[
\Phi_t =
\begin{pmatrix}
0 & \cdots &0& 1	\\
0 & \cdots & -1 & 0 \\
 \vdots & \vdots & \vdots & \vdots \\
\left(-1\right)^{t-1} & \cdots & 0 & 0
\end{pmatrix}.
\]

\subsection{Unramified places}
If $\nu \not\in S$ is a non-archimedean place of $F$, then $H_\nu$ and $G_\nu$ are unramified groups and $\xi$ induces a map from the $K_{H,\nu}$-unramified representations of $H(F_\nu)$ to the $K_{G,\nu}$-unramified representations of $G(F_\nu)$.  Dual to this transfer, there exists a morphism of  spherical Hecke algebras 
(cf. \cite[\S2.7]{minguez})
\[
	\xi^*_{\mathrm{ur}, \nu} : \mathcal{C}_c^\infty(K_{G,\nu} \backslash G(F_\nu)/K_{G,\nu}, \qpbar)
	\rightarrow \mathcal{C}_c^\infty(K_{H,\nu} \backslash H(F_\nu)/K_{H,\nu}, \qpbar).
\]
Combining these morphisms, we obtain the morphism of spherical Hecke algebras
\[
	\xi^*_{\mathrm{ur}} = \otimes_{\nu\not\in S \cup S_p} \xi^*_{\mathrm{ur}, \nu} : \mathcal{H}_{G, \mathrm{ur}}   \rightarrow \mathcal{H}_{H,\mathrm{ur}}.
\]	
 
 \subsection{Places in $S_p$}
 Let $\nu \in S_p$. 
 We shall slightly modify the construction of the morphism of the spherical Hecke algebras (cf. \cite[\S2]{minguez})
 to obtain a map of refinements. 
 
The $\qpbar$-valued characters of 
 $T_{H,\nu}/ T^0_{H,\nu}$ (resp. $T_{G,\nu}/ T^0_{G,\nu}$) are naturally parameterised by the $\qpbar$-valued points
 of $\widehat{T}_{H,\nu}$ (resp. $\widehat{T}_{G,\nu}$) where
 $\widehat{T}_{H,\nu}$ (resp. $\widehat{T}_{G,\nu}$) denotes the dual torus of $T_{H,\nu}$ (resp. $T_{G,\nu}$).
 This is seen via the following chain of canonical bijections
 \begin{eqnarray*}
 \widehat{T}_{H,\nu}(\qpbar) &=& \mathrm{Hom}( X^*(\widehat{T}_{H,\nu}), \qpbar^\times )	\\
 		&=&  \mathrm{Hom}( X_*({T}_{H,\nu}), \qpbar^\times)	\\
		&=& \mathrm{Hom}(T_{H,\nu}/T^0_{H,\nu}, \qpbar^\times) \\
		&=& \mathrm{Hom}(T_{H,\nu}^{\overline{\omega}}, \qpbar^\times) 
 \end{eqnarray*}
 where $X^*$ (resp. $X_*$) denotes the group of algebraic characters (resp. algebraic co-characters) of the corresponding algebraic torus.  
 The first bijection is simply  the definition of the $\qpbar$-points of
$\widehat{T}_{H,\nu}$.  
The second bijection follows from the canonical bijection $X^*(\widehat{T}_{H,\nu}) = X_*(T_{H,\nu})$.
The third bijection is induced from the canonical bijection
\begin{eqnarray*}
	X_*(T_{H,\nu}) &\rightarrow& T_{H,\nu}/T_{H,\nu}^0	\\
	\alpha^\vee &\mapsto& \alpha^\vee(\overline{\omega}).
\end{eqnarray*}
The fourth bijection follows from the canonical bijection $T_{H,\nu}/T^0_{H,\nu} \simeq T_{H,\nu}^{\overline{\omega}}$.
  
The $L$-homomorphism $\xi : {}^L H \rightarrow {}^L G$ restricts to give a map
$\xi: \widehat{H} \times \mathcal{F}_\nu \rightarrow \widehat{G} \times \mathcal{F}_\nu$
where $\mathcal{F}_\nu$ refers to the geometric Frobenius element of $W_{F_\nu}$.
 This induces a map
 \begin{equation*}
 	 \widehat{T}_{H,\nu}\left(\qpbar\right)  \rightarrow
		\widehat{T}_{G,\nu}\left(\qpbar\right) .
 \end{equation*}
This map is canonical up to composing with an isomorphism of the form
\begin{eqnarray*}
\widehat{T}_{G,\nu}(\qpbar) &\rightarrow& \widehat{T}_{G,\nu}(\qpbar) 	\\
\diag(x_1,\ldots, x_m) &\mapsto& \diag(x_{\sigma(1)}, \ldots, x_{\sigma(m)})	\\
\end{eqnarray*}
where $\sigma \in \mathfrak{S}_m$ is a permutation.

 Dual to the map  $\widehat{T}_{H,\nu}\left(\qpbar\right)  \rightarrow \widehat{T}_{G,\nu}\left(\qpbar\right)$,  we have a morphism of $\qpbar$-algebras
   \[
 	\mathcal{R}_{\nu} :  \qpbar[T_{G,\nu}^{\overline{\omega}}] \rightarrow \qpbar[T_{H,\nu}^{\overline{\omega}}]
 \]
 which is canonical up to precomposing with an isomorphism of the form
\begin{eqnarray*}
\iota_\sigma : {T}_{G,\nu} &\rightarrow& {T}_{G,\nu} 	\\
\diag(x_1,\ldots, x_m) &\mapsto& \diag(x_{\sigma(1)}, \ldots, x_{\sigma(m)})	\\
\end{eqnarray*}
where $\sigma \in \mathfrak{S}_m$ is a permutation.
We shall refer to $\mathcal{R}_\nu$ as the \emph{refinement map}.
 \begin{remark}
 The reason that the morphism $\mathcal{R}_\nu$ is non-canonical
 is due to  the fact that we are working at the level of refinements whilst the classical Langlands correspondence
 operates at the level of representations.   Consequently in order to obtain a map of refinements, one is obliged to specify an ordering of the refinement map.
 Two such maps differ by $\iota_\sigma$ for 
 a choice of $\sigma \in \mathfrak{S}_m$.
 \end{remark}
  \begin{lemma}
 Let $\pi_{H,\nu}$ and $\pi_{G,\nu}$ be irreducible admissible representations of $H\left(F_\nu\right)$ and
 $G\left(F_\nu\right)$ respectively.
 Assume that
   $\pi_{H,\nu}^{I_{H,\nu}} \neq 0$, $\pi_{G,\nu}^{I_{G,\nu}} \neq 0$, and
 $\pi_{G,\nu}$ is the Langlands functorial transfer of  $\pi_{H,\nu}$ via the $L$-homomorphism $\xi$.
 If $\chi_{H,\nu} : T_{H,\nu}^{\overline{\omega}} \simeq T_{H,\nu}/T^0_{H,\nu} \rightarrow \qpbar^\times$ is a refinement of $\pi_{H,\nu}$, then $\chi_{H,\nu} \circ \mathcal{R}_\nu: T_{G,\nu}^{\overline{\omega}} \simeq T_{G,\nu}/T^0_{G,\nu} \rightarrow \qpbar^\times$ is a refinement of $\pi_{G,\nu}$.
 \end{lemma}
   \begin{remark}
 We stress that there is no claim that accessible refinements are mapped to accessible refinements.  Such a claim would be false in general.
 \end{remark}
 \begin{proof}
Writing as follows the unramified character
 \[
 \chi_{H, \nu} = \chi_{H, 1, \nu} \times \cdots \chi_{H, r, \nu} : T^{\overline{\omega}}_{GL_{n_1}, \nu} \times \cdots \times T^{\overline{\omega}}_{GL_{n_r}, \nu}
 	\rightarrow \qpbar^\times,
 \]
we observe that the semi-simple conjugacy class of $GL_{n_1} \times \cdots \times GL_{n_r}(\qpbar)$ corresponding
 to
 the geometric Frobenius 
  via the local Langlands correspondence
 for $\pi_{H,\nu}$ is represented by the element
 \[
\mathrm{diag}(\chi_{H,1,\nu}({\overline{\omega}}, 1, \ldots, 1), \ldots, \chi_{H,1,\nu}(1,\ldots, 1, \overline{\omega})) \times
\cdots \times
\mathrm{diag}(\chi_{H, r,\nu}({\overline{\omega}}, 1, \ldots, 1), \ldots, \chi_{H, r,\nu}(1,\ldots, 1, \overline{\omega})). 
 \]
 Our construction of the morphism $\mathcal{R}_\nu$ is based upon a slight modification of the
construction of the 
corresponding morphism of the spherical Hecke algebras $\xi^*_{\mathrm{ur}, \nu}$ (cf. \cite[\S2]{minguez}).  
It follows that the semi-simple conjugacy class of 
$GL_m(\qpbar)$
corresponding
 to
 the geometric Frobenius 
  via the local Langlands correspondence
 for $\pi_{G, \nu}$ is represented by the element
 \[
\mathrm{diag}(\chi_{G, \nu}({\overline{\omega}}, 1, \ldots, 1), \ldots, \chi_{G, \nu}(1,\ldots, 1, \overline{\omega}))
 \]
where $\chi_{G,\nu} =   \chi_{H,\nu} \circ \mathcal{R}_\nu: T^{\overline{\omega}}_{GL_{m}, \nu} \rightarrow \qpbar^\times$.
That is $\chi_{G, \nu}$ is a refinement of $\pi_{G,\nu}$.
 \end{proof}
  It will be convenient for  reasons of algebraicity arising in the construction of the eigenvariety to renormalise the morphism as follows
  \begin{eqnarray*}
   	\mathcal{R}'_\nu : \qpbar[T_{G,\nu}^{\overline{\omega}}] &\rightarrow& \qpbar[T_{H,\nu}^{\overline{\omega}}] 	\\
	t & \mapsto&  \mathcal{R}_\nu(t) \cdot \delta^{-1/2}_{B_{G,\nu}}(t) \cdot 
			(\delta^{1/2}_{B_{H,\nu}} \circ \mathcal{R}_\nu)(t).
  \end{eqnarray*}
  The effect of the normalisation is described by the following trivial lemma.
    \begin{lemma}
 If $\chi_{H,\nu} : T_{H,\nu}^{\overline{\omega}} \rightarrow \qpbar^\times$ is a homomorphism,  then
 the following maps are equal
 \[
  (\chi_{H,\nu} \cdot \delta_{B_{H,\nu}}^{-1/2}) \circ \mathcal{R}'_\nu  =  
 	(\chi_{H, \nu} \circ \mathcal{R}_\nu) \cdot \delta_{B_{G, \nu}}^{-1/2}
 : T^{\overline{\omega}}_{G,\nu} \rightarrow \qpbar^\times.
 \]
 \end{lemma}

 \subsection{Weight space map}
It remains to study the behaviour at archimedean places.
  We shall 
 assume that our $L$-homomorphism $\xi$ satisfies the following hypothesis, which includes
 a compatibility condition between the behaviour of $\xi$ at the non-archimedean places in $S_p$ and
 the archimedean places.
 
 \begin{hypothesis}
   \label{hypothesis_1}
If $\nu \in S_p$, then there exist  $\qpbar$-algebra homomorphisms
\begin{eqnarray*}
	\Xi^*_{\mathcal{W}, \nu} : \qpbar[T_{G,\nu}] \rightarrow \qpbar[T_{H,\nu}]	\\
	\Lambda^*_{\nu} : \qpbar[T^{\overline{\omega}, -}_{G,\nu}] \rightarrow \qpbar[T^{\overline{\omega}, -}_{H,\nu}]
\end{eqnarray*}
such that the following conditions are satisfied.
\begin{itemize}
	\item  
	Let $\omega =\iota_p^{-1} \nu$ be the archimedean place corresponding to $\nu$.  Let $\pi_{H,\omega}$, and
		$\pi_{G, \omega}$ be irreducible admissible representations of $H(F_\omega)$ and $G(F_\omega)$ respectively such that $\pi_{G,\omega}$  corresponds to $\pi_{H, \omega}$ via  Langlands functoriality for $\xi$.
		Then there exists a $\sigma \in \mathfrak{S}_m$ (depending upon $\pi_{G,\omega}$ and $\pi_{H,\omega}$) such that the 
		following two characters are equal
\[
		\kappa(\pi_{G, \omega}) =
		 \kappa(\pi_{H, \omega})  \circ \Xi^*_{\mathcal{W}, \nu} \circ \iota_\sigma: T_{G,\nu} \rightarrow \qpbar^\times.	\\
\]	
	Furthermore
	$\Xi^*_{\mathcal{W},\nu}$ restricts to give a $\mathbf{Z}_p$-algebra homomorphism
		\[
	\Xi^*_{\mathcal{W}, \nu}: \mathbf{Z}_p[T^0_{G,\nu}] \rightarrow \mathbf{Z}_p[T^0_{H,\nu}].
	\]
	
	\item 
	$\Xi^*_{\mathcal{W},\nu}$ restricts to give a $\qpbar$-algebra homomorphism 
			\[
	\Xi^*_{\mathcal{W}, \nu}: \qpbar[T^{\overline{\omega}}_{G,\nu}] \rightarrow \qpbar[T^{\overline{\omega}}_{H,\nu}]
	\]
	such that
	for all homomorphisms  $\chi, \delta : T^{\overline{\omega},-}_{H, \nu} \rightarrow \qpbar^\times$, 
	\[
	(\chi \cdot \delta)  \circ \Lambda^*_\nu = (\chi \circ \mathcal{R}'_\nu) \cdot (\delta \circ \Xi^*_{\mathcal{W},\nu}): T^{\overline{\omega}, -}_{G, \nu} \rightarrow \qpbar^\times.
	\]
\end{itemize}
 \end{hypothesis}

 If $\nu \in S_p$, then the morphism $\Xi^*_{\mathcal{W}, \nu}$ induces a $\mathbf{Z}_p$-algebra homomorphism 
  \[
  	\Xi^*_{\mathcal{W}, \nu} : \mathbf{Z}_p[[T^0_{G,\nu}]] \rightarrow \mathbf{Z}_p[[T^0_{H,\nu}]].
  \]
 This induces a $\qp$-rigid analytic morphism (cf. \cite[\S7]{dejong_article_ihes})  
  \[
 	\Xi_{\mathcal{W}, \nu} : \mathcal{W}_{H,\nu}  \rightarrow \mathcal{W}_{G,\nu}.
 \]
 Together these morphisms induce a $\qp$-rigid analytic morphism of the weight spaces
 \[
 	\Xi_{\mathcal{W}} = \prod_{\nu\in S_p} \Xi_{\mathcal{W},\nu} : \mathcal{W}_H \rightarrow \mathcal{W}_G.
 \]
 We also note that the morphisms $\Lambda^*_{\nu} : \qpbar[T^{\overline{\omega}, -}_{G,\nu}] \rightarrow \qpbar[T^{\overline{\omega}, -}_{H,\nu}]$ 
 where $\nu \in S_p$
 induce morphisms of the Atkin-Lehner algebras
 \begin{eqnarray*}
 	\Lambda^*_{\nu} : \mathfrak{U}^-_{G,\nu} \rightarrow \mathfrak{U}^-_{H,\nu}	\\
	 	\Lambda^* = \otimes_{\nu\in S_p} \Lambda^*_\nu : \mathfrak{U}^-_{G} \rightarrow \mathfrak{U}^-_{H}.	\\
 \end{eqnarray*}

\subsection{Definitions} 
 We are now in a position to  introduce a notion of $p$-adic functoriality.
 \begin{definition}
 A rigid analytic morphism
 \[
 	\Xi : \mathcal{D}_H \rightarrow \mathcal{D}_G
 \]
 is said to be the \emph{$p$-adic Langlands functoriality morphism} for the tuple  $(\xi, \Xi_{\mathcal{W}}, \Lambda^*)$
 (or simply $\xi$ if the context is clear) if 
 the following diagrams commute.
 \[
 \xymatrix{
\mathcal{D}_H \ar[d]^{\kappa_H}  \ar[r]^{\Xi} &  \mathcal{D}_G \ar[d]^{\kappa_G}	 & 
	 \mathcal{H}^-_{G} \ar[d]^{\psi_G}  \ar[r]^{ \Lambda^* \otimes \xi^*_{\mathrm{ur}}} &  \mathcal{H}^-_{H} \ar[d]^{\psi_H}	  \\
\mathcal{W}_H \ar[r]^{\Xi_{\mathcal{W}}} & \mathcal{W}_G & 
	\mathcal{O}\left(\mathcal{D}_G\right) \ar[r]^{\Xi^*} & \mathcal{O}\left(\mathcal{D}_H\right)
 }
 \]
 
 \end{definition}
 
 \begin{lemma}
 \label{lemma_if_padic_func_exists_it_is_unique}
 If the rigid analytic morphism $\Xi$ exists, then it is unique.
 \end{lemma}
 \begin{proof}
 Since the eigenvarieties $\mathcal{D}_H$ and $\mathcal{D}_G$ are reduced, the morphism $\Xi$ is completely determined
 by the induced map on the $\qpbar$-valued points.  The morphisms $\Xi_{\mathcal{W}}$ and ${\Lambda^* \otimes \xi^*_{\mathrm{ur}}}$ induce a map
 \[
 	\mathrm{Hom}(\mathcal{H}^-_H, \qpbar) \times \mathcal{W}_H(\qpbar) \rightarrow 
	\mathrm{Hom}(\mathcal{H}^-_G, \qpbar) \times \mathcal{W}_G(\qpbar)
 \]
 which fixes the 
  map of the eigenvarieties on the $\qpbar$-valued points.  The result follows.
 \end{proof}
 \begin{remark}
 Let us emphasise the main difference between classical Langlands functoriality and $p$-adic Langlands functoriality.
 In the $p$-adic setting, one must make a non-canonical choice for the map of refinements $\mathcal{R}_\nu$ for $\nu \in S_p$, and different choices of maps give rise to different notions of functoriality.
 \end{remark}
 
 Concerning compatibility with classical Langlands functoriality, we have the following result.
  \begin{lemma}
  Assume that $\Xi : \mathcal{D}_H \rightarrow \mathcal{D}_G$, the $p$-adic Langlands functoriality morphism for $\xi$, exists.
  Let $\pi_H$  (resp. $\pi_G$) be an automorphic representations of $H$  (resp. $G$)  
 such that $e_H(\pi_{H,f}) \neq 0$ (resp.  $e_G(\pi_{G,f}) \neq 0$)
 equipped with an accessible
 refinement $\chi_H$ (resp. $\chi_G$).  Let $x_H \in D_H(\qpbar)$ (resp. $x_G \in D_G(\qpbar)$) be
 the $\qpbar$-point on the eigenvariety corresponding to
 the pair 
 $(\nu_H(\pi_H, \chi_H) \otimes \psi_{H, \mathrm{ur}}(\pi_H), \kappa_H(\pi_{H,\infty}) )$
(resp. 
 $(\nu_G(\pi_G, \chi_G) \otimes \psi_{G, \mathrm{ur}}(\pi_G), \kappa_G(\pi_{G,\infty}) )$).
 Then the following statements hold.
 \begin{itemize}
 
 \item    If $\Xi(x_H) = x_G$, then
 $\pi_G$ corresponds to $\pi_H$ via  Langlands functoriality for $\xi$ at both the archimedean places and the  non-archimedean places $\nu \not\in S \cup S_p$.
  
  \item
  Conversely, if
 \begin{itemize}
 	\item  $\pi_G$ corresponds to $\pi_H$ via  Langlands functoriality for $\xi$ at the non-archimedean places $\nu \not\in S \cup S_p$;
	\item $\Xi_{\mathcal{W}}(\kappa(\pi_{H,\infty})) = \kappa(\pi_{G,\infty}) \in \mathcal{W}_{G}(\qpbar)$; and
	\item for  all $\nu \in S_p$,
  	$\chi_{H,\nu}  \circ \mathcal{R}_{\nu} = \chi_{G,\nu}: T^{\overline{\omega}}_{G,\nu} \rightarrow \qpbar^\times$,	\\

 \end{itemize}
 then 
  $\Xi(x_H) = x_G$.
   \end{itemize}
 \end{lemma}
 \begin{proof}
 We shall prove the first statement since the proof of the second statement  follows  from the definitions in a similar fashion.  The conditions imposed upon $\pi_H$ and $\pi_G$ ensure that the points
 $x_H$ and $x_G$ exist.   
  The representation $\pi_G$ is seen to be the  transfer of $\pi_H$ via Langlands functoriality at both the
  archimedean places and the non-archimedean places $\nu\not\in S \cup S_p$.  The former follows from our conditions
  imposed upon the weight space morphism $\Xi_{\mathcal{W}}$ whilst the later follows from the properties of the 
  spherical Hecke algebra morphism $\xi^*_{\mathrm{ur}} : \mathcal{H}_{G, \mathrm{ur}} \rightarrow \mathcal{H}_{H, \mathrm{ur}}$.

 \end{proof}
 
 \section{$p$-adic Langlands functoriality: Construction of the morphism}
 \label{p_adic_functoriality_examples}
 We shall construct the $p$-adic Langlands functoriality morphism when the $L$-homomorphism $\xi$ is the Langlands direct sum. 
 
 \subsection{The Langlands direct sum for the unitary group}
 We shall define the Langlands direct sum $L$-homomorphism and make explicit the induced transfer of representations.
 Let $E/F$ be a totally imaginary quadratic extension of a totally real field.
Let $\eta: F^\times\backslash\mathbf{A}^\times_F \rightarrow \mathbf{C}^\times$
be the character 
associated to the  field extension $E/F$ via class field theory.
Fix a  unitary character
\[
\mu : E^\times\backslash \mathbf{A}_E^\times \rightarrow \mathbf{C}^\times
\]
 such that
\begin{itemize}
\item	$\mu$ extends $\eta$;
\item if $\nu \in S_p$, then $\mu$ is unramified at $\nu$; and
\item if $\nu$ is non-archimedean and inert in $E$, then $\mu$ is unramified at $\nu$.
\end{itemize}
 The character $\mu$ can be seen via class field theory as a character of the Weil group
 $W_{E}$.  
 At archimedean places $\omega$ of $E$, the Hecke character 
is of the  form (cf. \cite[\S6.9.2]{bellaiche_chenevier_book})
		\begin{align*}
			\mu_{\omega} : \mathbf{C}^\times &\rightarrow \mathbf{C}^\times \\
			z &\mapsto \left(z/\overline{z}\right)^{\alpha_{\omega}}
	\end{align*}
for some half-integer $\alpha_\omega$. 

 For all integers $i \in \mathbf{Z}$,  we shall  define  the Hecke character
$\mu_i : E^\times\backslash \mathbf{A}_E^\times \rightarrow \mathbf{C}^\times$
\[
	\mu_{i} = \begin{cases}
		\mu &: \text{if $i$ is odd}	\\
		\mathbf{1} &: \text{if $i$ is even}
	\end{cases}
\]
and correspondingly for all archimedean places $\omega$ of $E$, we define
\[
	\alpha_{i, \omega} = \begin{cases}
		\alpha_\omega	&: \text{if $i$ is odd}	\\
		0	&: \text{if $i$ is even}
	\end{cases}
\]
In what follows, we shall often abuse notation and view an archimedean place $\omega$ of $F$ as the  
corresponding
archimedean place of $E$ lying above $\omega$ and vice versa.  We shall also often view the characters
$\mu_i$ as having values in $\qpbar$ via the isomorphism $\iota_p$.

We can now introduce the following $L$-homomorphism, which we shall refer to as the \emph{Langlands direct sum}
$L$-homomorphism.
\begin{eqnarray*}
	\xi = \xi_{\vec{n}} : {}^L (U_{n_1} \times \cdots \times  U_{n_r}) &\rightarrow& {}^L U_n		\\
		g_1 \times \cdots \times g_r \times 1 &\mapsto& \mathrm{diag}\left(g_1,\ldots, g_r\right) \times 1	\\
		I_{n_1} \times \cdots I_{n_r} \times w &\mapsto& \mathrm{diag}\left(\mu_{n-n_1}\left(w\right) I_{n_1}, \ldots, \mu_{n-n_r}\left(w\right)  I_{n_r} 	\right) \times w \ \  \forall w \in W_E	\\
		I_{n_1} \times \cdots \times I_{n_r} \times w_c &\mapsto& \mathrm{diag}\left(\Phi_{n_1}, \ldots, \Phi_{n_r}\right) \Phi_n^{-1} \times w_c
\end{eqnarray*}
where $w_c$ denotes a chosen lift of $c \in \mathrm{Gal}\left(E/F\right)$
and $I_{n_i}$ denotes the identity matrix of $GL_{n_i}(\mathbf{C})$.
\begin{remark}
If $\vec{n} = \left(a,b\right)$, then $\xi_{\vec{n}}$ is the endoscopic $L$-homomorphism studied in \cite[\S4]{rogawski_book}.
\end{remark}
 We shall now make the corresponding Langlands functorial transfer of representations
 explicit in some cases (cf. \cite[\S4.2]{white_endoscopy_unitary_group}).
 \begin{itemize}
 	\item Assume that $\omega|\infty$.
	Let $\pi_\omega$ be an irreducible admissible representation of 
	$U_{\vec{n}}\left(F_\omega\right)$.
Let $\{ k_{i,j} : 1 \leq i \leq r,\ 1 \leq j \leq n_i\}$ be the ordered set of highest weights of $\pi_\omega$.	
	The  Langlands functorial transfer of $\pi_\omega$ is defined if the numbers
	\[
		k_{i,j} + \frac{n_i +1}{2} -j + \alpha_{n-n_i, \omega},  \ i=1,\ldots, r \ \ j=1,\ldots, n_i
	\]
	 are distinct (otherwise the corresponding $L$-parameter of $U_n(F_\omega)$ will not be relevant).
	 Let 
	 \[
	 	k'_i,  \ i=1,\ldots, n
	 \]
	 denote the reordering of the 
	 $k_{i,j} + \frac{n_i +1}{2} -j + \alpha_{n-n_i, \omega}$
	  such that $k'_s > k'_{t}$ for all $s > t$, and
	  let $k_i = k'_i - \frac{n+1}{2} + i$ for all $i=1,\ldots, n$.
The Langlands functorial transfer of $\pi_\omega$, denoted $\xi_{\vec{n}}\left(\pi_\omega\right)$, is the irreducible
	admissible representation of $U_{n}\left(F_\omega\right)$ whose highest weights  are equal to 
	\[
		\{k_i : i=1,\ldots, n\}
	\]
 We remind the reader that the terms $\frac{n_i + 1}{2} - j$ and $-\frac{n+1}{2} + i$ appear due to the difference between the Langlands parameterisation and the highest weight parameterisation of an irreducible admissible representation
 of $U(F_\omega)$ (cf. \cite[\S6.7]{bellaiche_chenevier_book}).
	
	\item Assume that $\nu$ is a finite place of $F$ that splits in $E$. If $\pi_\nu$ is an irreducible admissible unitary representation of
	$U_{\vec{n}}\left(F_\nu\right) = GL_{n_1} \times \cdots GL_{n_r}(F_\nu)$, then
	\[
		\mathrm{Ind}_{P_{\vec{n}}}^{U_n} \pi_{1,\nu} \cdot \mu_{n-n_1,\nu} \times \cdots
			\times \pi_{r,\nu} \cdot \mu_{n-n_r, \nu}
	\]
	where $P_{\vec{n}}$ denotes the standard Parabolic subgroup with Levi-component $U_{\vec{n}}$,
	is an irreducible admissible unitary representation of $U_n(F_\nu)$ (cf. \cite{bernstein_irreducble_induction}), and it is the Langlands
	functorial transfer of $\pi_\nu$.
	
	\item Assume that $\nu$ is a finite place of $F$ that remains inert in $E$.  If $\pi_\nu$ is unramified, then
	the correspondence can be explicitly described in terms of Satake parameters
	(cf. \cite[\S4]{minguez}).
 \end{itemize}

Let $\nu \in S_p$.
We shall  study the map on refinements induced by $\mathcal{R}_\nu$.
Explicitly, the morphism can be seen to be equal to
 $\mathcal{R}_\nu =  \mathcal{R}_{0,\nu} \circ \iota_{\sigma_\nu}$ for a choice of $\sigma_\nu \in \mathfrak{S}_n$ where 
 \begin{eqnarray*}
\mathcal{R}_{0,\nu}: \qpbar[T_{U_{{n}}, \nu}^{\overline{\omega}}]
	&\rightarrow& \qpbar[T_{U_{\vec{n}}, \nu}^{\overline{\omega}}]	\\
\diag(x_{1,1}, \ldots, x_{1,n_1}, x_{2,1}, \ldots, x_{r, n_r}) 
	&\mapsto& \mu_{n-n_1,\nu}(x_{1,1}  \cdots x_{1,n_1})\cdot \diag(x_{1,1}, \ldots, x_{1,n_1}) \times \\
	&& \ \ \cdots \times
	\mu_{n-n_r,\nu}(x_{r,1}  \cdots x_{r,n_r}) \cdot
			\diag(x_{r,1}, \ldots, x_{r,n_r}).
\end{eqnarray*}
The $\qpbar$-algebra morphism $\mathcal{R}_{0,\nu}$ induces a map on refinements, which can be
explicitly described as follows.
Let 
\[
\chi_{U_{\vec{n}, \nu}}
= \chi_{U_{\vec{n}},1, \nu} \times \cdots \times \chi_{U_{\vec{n}},r, \nu}: T^{\overline{\omega}}_{U_{\vec{n}},\nu} \simeq T_{U_{\vec{n}},\nu}/T^0_{U_{\vec{n}},\nu} \rightarrow \qpbar^\times
\]
be an unramified character.
Then
\begin{eqnarray*}
 \chi_{U_{\vec{n}, \nu}} \circ \mathcal{R}_{0,  \nu}  : T_{U_{{n}},\nu}/T^0_{U_{{n}},\nu} \rightarrow \qpbar^\times	\\
\diag(x_{1,1}, \ldots, x_{1,n_1}, x_{2,1}, \ldots, x_{r, n_r}) 
	&\mapsto&
		\prod_{i=1}^r 
	 \mu_{n-n_i,\nu}(x_{i,1}  \cdots x_{i,n_i})\cdot \chi_{U_{\vec{n}}, i, \nu}(\diag(x_{i,1}, \ldots, x_{i,n_i})).  \\
\end{eqnarray*}
Under certain conditions, one can ensure that accessible refinements are mapped to accessible refinements, which will
not be the case in general.
\begin{lemma}
\label{lemma_accessible_refinements_mapped_accessible_refinements}
Let $\nu \in S_p$.
  Let $\pi_{U_{\vec{n}},\nu} = \times_{i=1}^r \pi_{U_{\vec{n}}, i, \nu}$ be an irreducible admissible tempered representation
of $U_{\vec{n}}(F_\nu)$ such that $\pi_{U_{\vec{n}},\nu}^{I_{{U_{\vec{n}},\nu}}} \neq 0$.
  Let 
$\chi_{U_{\vec{n}}, \nu}$ be an accessible refinement of $\pi_{U_{\vec{n}, \nu}}$.
Let $\pi_{U_{n}, \nu}$ be the 
$\xi_{\vec{n}}$-Langlands functorial transfer of $\pi_{U_{\vec{n}},\nu}$ to $U_{n}(F_\nu)$.  
Let $\sigma_\nu \in \mathfrak{S}_n$ such that 
\[
\sigma_\nu(i) < \sigma_\nu(j) \textrm{ for all } n_0 + \cdots + n_{t} \leq i < j \leq n_{t+1} \textrm{ for all } 0 \leq t \leq r-1
\]
where $n_0 = 0$.
Then $\pi_{U_{{n}},\nu}^{I_{U_{{n}},\nu}} \neq 0$
and $\chi_{U_{\vec{n}}, \nu} \circ  \mathcal{R}_{0, \nu} \circ \iota_{\sigma_\nu}$ is an accessible refinement
of $\pi_{U_n, \nu}$.
\end{lemma}
\begin{proof}
Since the representation $\pi_{U_{\vec{n}}, \nu}$ is tempered with an Iwahori-invariant vector, it is 
isomorphic to
a representation of the form
\[
	\pi_{U_{\vec{n}}, \nu} \simeq 
		\times_{i=1}^r \text{Ind}_{P_i}^{GL_{n_i}} \gamma_{i,1} \cdot \text{St}(d_{i,1}) \times \cdots \times \gamma_{i,m_i} \cdot \text{St}(d_{i,m_i})
\]
where for all $i=1,\ldots,r$
\begin{itemize}
\item $\sum_{j=1}^{m_i} d_{i,j} =n_i$;
\item $M_i$ is the standard Levi subgroup with blocks of size $d_{i,j}$ and
 $P_i$ is the corresponding standard Parabolic; and
\item for all $j=1,\ldots, m_i$, 
	\begin{itemize}
	\item if $d_{i,j} \neq 1$, then
 $\gamma_{i,j}$ is an unramified unitary character
  and $\text{St}(d_{i,j})$ denotes the standard Steinberg representation of $GL_{d_{i,j}}(F_\nu)$ and
  \item
  if $d_{i,j} = 1$, then $\gamma_{i,j} \neq \mathbf{1}$ 
  is a non-trivial unramified unitary character
  and  $\text{St}(1) = \mathbf{1}$ denotes  the trivial character.
  \end{itemize}
\end{itemize}
It follows that
\[
	\pi_{U_n, \nu} \simeq
			 \text{Ind}_{P}^{GL_{n}}
			 \times_{i=1}^r \times_{j=1}^{m_i}
			\mu_{n-n_i,\nu} \cdot  \gamma_{i,j} \cdot \text{St}(d_{i,j}) 
		 \] 
where $P$ denotes the corresponding standard Parabolic.
This representation has an Iwahori invariant vector. The fact that
the morphism $\mathcal{R}_\nu$, under the restrictions imposed upon $\sigma_\nu$, sends
 an accessible refinement
of $\pi_{U_{\vec{n}}, \nu}$
to an accessible refinement of $\pi_{U_n, \nu}$  follows from
a general result of Zelevinsky \cite[Theorem 6.1]{zelevisnsky}.
\end{proof}

Let $\nu \in S_p$ and let $\omega = \iota_p^{-1}\circ  \nu \in S_\infty$ denote the corresponding archimedean place. 
Choose a $\sigma_\nu \in \mathfrak{S}_n$ as in Lemma \ref{lemma_accessible_refinements_mapped_accessible_refinements}.
The morphisms $\Lambda^*_\nu$ and $\Xi^*_{\mathcal{W},\nu}$ can be chosen to be
 $\Lambda^*_\nu =   \Lambda^*_{0,\nu} \circ \iota_{\sigma_\nu}$ and
	$\Xi^*_{\mathcal{W},\nu} = \Xi^*_{\mathcal{W},0, \nu} \circ \iota_{\sigma_\nu}$
where
\begin{eqnarray*}
\Xi^*_{\mathcal{W},0, \nu} : \qpbar[T_{U_{{n}}, \nu}] &\rightarrow& \qpbar[T_{U_{\vec{n}}, \nu}]		\\
\diag(x_{1,1}, \ldots, x_{1,n_1}, x_{2,1}, \ldots, x_{r, n_r}) 
	&\mapsto& \prod_{i=1}^r \prod_{j=1}^{n_i} x_{i,j}^{\alpha_{n-n_i, \omega} + \frac{n_i -n}{2}   + {n_1 + \cdots + n_{i-1}}} \diag(x_{i,1}, \ldots, x_{i,n_i}) \\
   	\Lambda^*_{0, \nu}  : \qpbar[T_{U_{n},\nu}^{\overline{\omega}, -}] &\rightarrow& \qpbar[T_{U_{\vec{n}},\nu}^{\overline{\omega}, -}] 	\\
	t & \mapsto&  \Lambda^{\dagger}_{0, \nu} (t) \cdot \delta^{-1/2}_{B_{U_n,\nu}}(t) \cdot 
			(\delta^{1/2}_{B_{U_{\vec{n}},\nu}} \circ \Lambda^{\dagger}_{0, \nu} )(t) \\
\Lambda^{\dagger}_{0, \nu} : \qpbar[T^{\overline{\omega}, -}_{U_{{n}}, \nu}] &\rightarrow& \qpbar[T^{\overline{\omega}, -}_{U_{\vec{n}}, \nu}]		\\
\diag(x_{1,1}, \ldots, x_{1,n_1}, x_{2,1}, \ldots, x_{r, n_r}) 
	&\mapsto& \prod_{i=1}^r \prod_{j=1}^{n_i}
	\mu_{n-n_i,\nu}(x_{i,1}  \cdots x_{i,n_i}) \\
&& \ \ \ \ \ \ \ \	 x_{i,j}^{\alpha_{n-n_i, \omega} + \frac{n_i -n}{2}   + {n_1 + \cdots + n_{i-1}}} \diag(x_{i,1}, \ldots, x_{i,n_i}). 
\end{eqnarray*}

\begin{lemma}
\label{lemma_weight_space_morphism_isomorphism}
The induced morphism 
\[
	\Xi_{\mathcal{W}}: \mathcal{W}_{U_{\vec{n}}} \rightarrow \mathcal{W}_{U_n}
\]
is an isomorphism of weight spaces.
\end{lemma}
\begin{proof}
This follows directly from the fact that the induced $\mathbf{Z}_p$-algebra morphisms
\[
	\Xi^*_{\mathcal{W}, \nu}: \mathbf{Z}_p[[T^0_{U_n, \nu}]] \rightarrow \mathbf{Z}_p[[T^0_{U_{\vec{n}}}]]
\]
are isomorphisms for all $\nu \in S_p$.
\end{proof}

\begin{lemma}
\label{lemma_langlands_direct_sum_after_fix_sigma_morphisms_weight_unique}
Let $\nu \in S_p$ and
let $\sigma_\nu \in \mathfrak{S}_n$ be as in Lemma \ref{lemma_accessible_refinements_mapped_accessible_refinements}
(this fixes 
 the morphism $\mathcal{R}_\nu =  \mathcal{R}_{0,\nu} \circ \iota_{\sigma_\nu}$).
The morphisms 
$\Lambda^*_\nu$ and $\Xi^*_{\mathcal{W},\nu}$ 
are
the unique morphisms that
satisfy Hypothesis \ref{hypothesis_1} for this choice of $\mathcal{R}_{\nu}$.
\end{lemma}
\begin{proof}
The fact that the morphisms satisfy Hypothesis \ref{hypothesis_1} follows from a simple check.  
To see that the morphisms are unique, we observe the following.  
By the first condition of Hypothesis \ref{hypothesis_1}, one observes that
the morphism
$\Xi^*_{\mathcal{W},\nu} =  \Xi^*_{\mathcal{W},0, \nu} \circ \iota_{\sigma_\nu}$ is uniquely determined up to
 composing with an isomorphism of the form $\iota_{\sigma'}$ where $\sigma' \in \mathfrak{S}_n$.
 The second condition of Hypothesis \ref{hypothesis_1} forces  $\sigma'$ to be the identity permutation, that is the morphism
 $\Xi^*_{\mathcal{W},\nu}$ is the unique morphism to satisfy
 Hypothesis \ref{hypothesis_1}.
 The uniqueness of the morphisms $\Lambda^*_\nu$  is   seen to follow
 from the second condition of Hypothesis \ref{hypothesis_1}.
 
 \end{proof}

 \subsection{$p$-adic Langlands functoriality for the Langlands direct sum}
 \label{subsection_p_adic_functorality_Langlands_direct_sum}
 For all $\nu \in S_p$, let $\sigma_\nu \in \mathfrak{S}_n$ as in Lemma \ref{lemma_accessible_refinements_mapped_accessible_refinements}.
 We shall consider data $(p, S, e_{U_{\vec{n}}}, \phi_{U_{\vec{n}}})$ and $(p, S, e_{U_n}, \phi_{U_n})$
of the following form.
\begin{itemize}
	\item $E/F$ is a totally imaginary quadratic extension of a totally real field that is unramified at all finite places.
	\item $\lambda \not\in S_p$ is a non-archimedean place of $F$ that does not split in $E$.
	\item $S$ is a finite set 
	of places of $F$ such that $\lambda \in S$, $S \cap S_p = \emptyset$, and if $\nu \in S - \lambda$ then
	$\nu$ is non-archimedean and splits in $E$.
		\item For all $\nu \in S - \lambda$, $e_{U_{\vec{n}},\nu} = \mathbf{1}_{I_{U_{\vec{n}},\nu}}$ 
		(resp.  $e_{U_n,\nu} = \mathbf{1}_{I_{U_n,\nu}} $)
	seen as an idempotent
			of the Hecke algebra $\mathbf{C}_c^\infty(U_{\vec{n}}(F_\nu), \qpbar)$ 
			(resp.  $\mathbf{C}_c^\infty(U_{n}(F_\nu), \qpbar)$).
	\item $e_{U_{\vec{n}}, \lambda}$ corresponds to the Bernstein component of a supercuspidal representation 
	(cf. \cite[Example 7.3.3]{bellaiche_chenevier_book})
	$\sigma_{U_{\vec{n}}, \lambda} = \sigma_{U_{\vec{n}}, 1,\lambda} \times \cdots \times \sigma_{U_{\vec{n}}, r,\lambda}$ such that 
	for all $i \neq j$,
	\[
		\sigma_{U_{\vec{n}}, i, \lambda} \not\simeq \sigma_{U_{\vec{n}}, j,\lambda}. 
	\]
	$e_{U_n, \lambda}$ is the  sum of the idempotents
	corresponding 
		 to the Bernstein components containing the discrete series representations
	$\sigma_{U_n, \lambda} \in \Pi$ where $\Pi$ is the $L$-packet of discrete series representations of $U_n(F_\lambda)$ which are the $\xi_{\vec{n}}$-Langlands functorial transfer of
	$\sigma_{U_{\vec{n}}, \lambda}$ (see \cite[\S3.3.5]{white_endoscopy_unitary_group} for a discussion of the  local Langlands classification of discrete series representations for the quasi-split unitary group due to M{\oe}glin \cite{moeglin_local_langlands_unitary_group}).
	
	\item $\phi_{U_{\vec{n}}} \in \qpbar[T_{U_{\vec{n}}, S_p}^{--}]$ and $\phi_{U_n} = \Lambda^*(\phi_{U_{\vec{n}}})$ which due to our explicit description 
	of $\Lambda^*$ we see lies in
	$\qpbar[T_{U_n,S_p}^{--}]$.
\end{itemize} 
We  shall now construct the $p$-adic Langlands functionality morphism 
$\Xi : \mathcal{D}_{U_{\vec{n}}} \rightarrow \mathcal{D}_{U_n}$ for the Langlands direct sum $\xi_{\vec{n}}$
and our choice of $\sigma_{\nu}$ for $\nu \in S_p$.
In order to do so, we must first introduce an auxiliary eigenvariety.

\subsubsection{Auxiliary eigenvariety}  
\label{subsubsection_auxiliary_eigenvarieties}
  The morphism
\[
	\Lambda^* \otimes \xi^*_{\mathrm{ur}} : \mathcal{H}^-_{U_{{n}}} \rightarrow \mathcal{H}^-_{U_{\vec{n}}}
\]
equips
the $p$-adic forms of type $(p, S, e_{U_{\vec{n}}}, \phi_{U_{\vec{n}}})$, 
denoted
 $\mathbf{S}_{U_{\vec{n}}}  = \{ S_{U_{\vec{n}}, i}; \iota_i : i \in \mathbf{N}\}$, with an action of the Hecke algebra
$\mathcal{H}^-_{U_{{n}}}$
\[
	\mathcal{H}^-_{U_n} \rightarrow \mathcal{H}^-_{U_{\vec{n}}} \rightarrow \mathrm{End}(\mathbf{S}_{U_{\vec{n}}})
\]
where the second morphism is the natural action of the Hecke algebra on $\mathbf{S}_{U_{\vec{n}}}$.  
Feeding the datum $(\mathcal{W}_{U_{\vec{n}}}, \mathbf{S}_{U_{\vec{n}}}, \mathcal{H}^-_{U_n}, \phi_{U_n})$ into
Buzzard's eigenvariety machine, we obtain the corresponding Eigenvariety
 $(\mathcal{D}', \psi', \kappa')$
which is seen to be reduced (cf. \cite[Proposition 3.9]{chenevier_jacquet_langlands_p_adic}).

\subsubsection{The two morphisms}
We shall construct 
the morphism 
$\Xi : \mathcal{D}_{U_{\vec{n}}} \rightarrow \mathcal{D}_{U_n}$ as the composite of two morphisms $\Xi_1$ and $\Xi_2$ where $\Xi_1 : \mathcal{D}_{U_{\vec{n}}} \rightarrow \mathcal{D}'$ 
and $\Xi_2 : \mathcal{D}' \rightarrow \mathcal{D}_{U_n}$ are $\qpbar$-rigid analytic morphisms
such that the following diagrams commute.

 \[
 \xymatrix{
\mathcal{D}_{U_{\vec{n}}} \ar[d]^{\kappa_{U_{\vec{n}}}}  \ar[r]^{\Xi_1} &  \mathcal{D'} \ar[ld]^{\kappa'}	 & 
	 \mathcal{H}^-_{U_n} \ar[d]^{\psi'}  \ar[r]^{ \Lambda^* \otimes \xi^*_{\mathrm{ur}}} &  \mathcal{H}^-_{U_{\vec{n}}} \ar[d]^{\psi_{U_{\vec{n}}}}	  \\
\mathcal{W}_{U_{\vec{n}}} &  & 
	\mathcal{O}\left(\mathcal{D}'\right) \ar[r]^{\Xi_1^*} & \mathcal{O}\left(\mathcal{D}_{U_{\vec{n}}}\right)
 }
 \]

 \[
 \xymatrix{
\mathcal{D}' \ar[d]^{\kappa'}  \ar[r]^{\Xi_2} &  \mathcal{D}_{U_n} \ar[d]^{\kappa_{U_n}}	 & 
	 \mathcal{H}^-_{U_n} \ar[d]^{\psi_{U_n}}  \ar[rd]^{\psi'}  & 	  \\
\mathcal{W}_{U_{\vec{n}}} \ar[r]^{\Xi_{\mathcal{W}}} & \mathcal{W}_{U_n} & 
	\mathcal{O}\left(\mathcal{D}_{U_n}\right) \ar[r]^{\Xi^*_2} & \mathcal{O}\left(\mathcal{D}'\right)
 }
 \]

\subsubsection{Existence of the  morphism $\Xi_1$}
\begin{lemma}
\label{lemma_existence_morphism_1}
There exists a $\qpbar$-rigid analytic morphism
$\Xi_1 : \mathcal{D}_{U_{\vec{n}}} \rightarrow \mathcal{D}'$ for which the
respective diagrams in Section \ref{subsubsection_auxiliary_eigenvarieties} commute.

\end{lemma}
\begin{proof}
In order to construct $\Xi_1$, we shall require Buzzard's construction of the eigenvariety which we shall briefly recall.  Details of this construction can be found in either \cite[\S7.3.6]{bellaiche_chenevier_book} or \cite{buzzard_eigenvariety}.

One associates to $\phi_{U_{\vec{n}}}$ a unique Fredholm power series
$P_{\phi_{U_{\vec{n}}}}(T) \in 1 + T \cdot \mathcal{O}(\mathcal{W}_{U_{\vec{n}}})\{\{T\}\}$ 
 such that for all open affinoids $V \subset \mathcal{W}_{U_{\vec{n}}}$ and for all $i \in \mathbf{N}$ sufficiently large
 \[
	P_{\phi_{U_{\vec{n}}}}(T)|_V = \det(1 - T \cdot \phi_{U_{\vec{n}}}|_{S_{U_{\vec{n}}, i}(V)}) \in 1 + T \cdot \mathcal{O}(V)\{\{T\}\}.
\]
Let  $Z(P_{\phi_{U_{\vec{n}}}}) \subset  \mathcal{W}_{U_{\vec{n}}} \times \mathbf{G}_m^{\mathrm{rig}}$ denote the associated Frodholm hypersurface.  
The eigenvariety $\mathcal{D}_{U_{\vec{n}}}$ is constructed as a cover for
 $Z(P_{\phi_{U_{\vec{n}}}})$.   
We shall write
 $\text{pr}_1 : Z(P_{\phi_{U_{\vec{n}}}}) \rightarrow \mathcal{W}_{U_{\vec{n}}}$ for the projection map onto the $\mathcal{W}_{U_{\vec{n}}}$ component.
There exists a canonical admissible covering $\mathcal{C}^*_{U_{\vec{n}}}$
of
 $Z(P_{\phi_{U_{\vec{n}}}})$ 
 given by the open affinoids $\Omega^* \subset  Z(P_{\phi_{U_{\vec{n}}}})$ for which the image
$\text{pr}_1(\Omega^*) \subset \mathcal{W}_{U_{\vec{n}}}$ is an open affinoid and the induced map
$\text{pr}_1|_{\Omega^*} : \Omega^* \rightarrow \text{pr}_1(\Omega^*)$ is finite.
If $\Omega^* \in \mathcal{C}^*_{U_{\vec{n}}}$ and  $V = \text{pr}_1(\Omega^*)$, then
one can use the resultant to canonically factorize 
\[
P_{\phi_{U_{\vec{n}}}}|_V  = Q_{\phi_{U_{\vec{n}}}}R_{\phi_{U_{\vec{n}}}}\in \mathcal{O}(V)\{\{T\}\}.
\]
This induces
for $i \in \mathbf{N}$ sufficiently large, an $\mathcal{O}(V)$-Banach module $\mathcal{H}_{U_{\vec{n}}}^-$-equivariant decomposition
\[
{S}_{U_{\vec{n}},i}(V) = {S}_{U_{\vec{n}}}(\Omega^*) \oplus {N}(\Omega^*, i)
\]
where ${S}_{\vec{n}}(\Omega^*)$ is a finite projective $\mathcal{O}(V)$-module that is independent of $i$.
The part of the eigenvariety lying above $\Omega^*$, that is $\nu_{U_{\vec{n}}}^{-1}(\Omega^*)$ is equal to the maximal spectrum of the image of $\mathcal{H}^-_{U_{\vec{n}}}$ in 
$\mathrm{End}({S}_{U_{\vec{n}}}(\Omega^*))$.

Consider now the analogous objects used in the construction of the eigenvariety $\mathcal{D}'$.
By construction, 
the underlying Banach spaces ${S}_{U_{\vec{n}},i}(V)$ are identical to those used in the construction of $\mathcal{D}_{U_{\vec{n}}}$ and the corresponding action of $\phi_{U_{{n}}}$ is equal to the action of $\phi_{U_{\vec{n}}}$.
It follows that the corresponding Fredholm power series are equal
and the
 associated Fredholm hypersurfaces are identical.

Let
$\Omega^* \in \mathcal{C}^*_{U_{\vec{n}}}$, let $V = \text{pr}_1(\Omega^*)$, and let $i\in \mathbf{N}$ be
sufficiently large.  Using the obvious notation since the Fredholm power series are identical, we see that 
$\Omega^* \in {\mathcal{C}^{*}}'$ and
\[
{S}_{U_{\vec{n}}}(\Omega^*) = S'(\Omega^*).
\]
Write $A'_{ \Omega^*}$ (resp. $A_{U_{\vec{n}}, \Omega^*}$) for the image of 
$\mathcal{H}^-_{U_n}$ (resp. $\mathcal{H}^-_{U_{\vec{n}}}$) in
$\mathrm{End}({S}_{U_{\vec{n}}}(\Omega^*))$.
Since $\mathcal{H}^-_{U_n}$ acts through the morphism 
$\Lambda^* \otimes \xi^*_{\mathrm{ur}} : \mathcal{H}^-_{U_{{n}}} \rightarrow \mathcal{H}^-_{U_{\vec{n}}}$,
we have a natural embedding $A'_{\Omega^*} \hookrightarrow A_{U_{\vec{n}}, \Omega^*}$.
This induces a morphism between the parts of the eigenvariety lying above 
 $\Omega^*$:
\[ \nu_{U_{\vec{n}}}^{-1}(\Omega^*) \rightarrow \nu'^{-1}(\Omega^*).\]
These morphisms glue together to give a morphism 
$\Xi_1 : \mathcal{D}_{U_{\vec{n}}} \rightarrow \mathcal{D}'$, which is  seen to make the
respective diagrams in Section \ref{subsubsection_auxiliary_eigenvarieties} commute.
\end{proof}

\subsubsection{Existence of the  morphism $\Xi_2$}
The morphism $\Xi_2$ shall be constructed by
applying some work of Chenevier \cite{chenevier_jacquet_langlands_p_adic} to
 interpolate the classical Langlands functoriality transfer for the Langlands direct sum.

\begin{lemma}
\label{lemma_classical_transfer_exists_result}
Let $\pi_{U_{\vec{n}}}$ be an  automorphic representation of $U_{\vec{n}}(\mathbf{A})$
such that
\begin{itemize}
	\item for all archimedean places $\nu$, $\pi_{U_{\vec{n}}, \nu}$ has regular highest weight and 
	the local $\xi$-Langlands functorial transfer to $U_n(F_\nu)$ exists and has regular highest weight, 
	\item for all non-archimedean places $\nu \neq \lambda$, $\pi_{U_{\vec{n}}, \nu}^{K_{U_{\vec{n}}, \nu}} \neq 0$, and
	\item $\pi_{U_{\vec{n}},\lambda} \simeq \sigma_{U_{\vec{n}}, \lambda}$.
\end{itemize}
Then,
\begin{itemize}
\item  $\pi_{U_{\vec{n}}}$ appears in the discrete automorphic spectrum of $U_{\vec{n}}(\mathbf{A})$ with multiplicity one,
\item  $\pi_{U_{\vec{n}}, \nu}$ is tempered for all non-archimedean places $\nu$ that split in $E$, 
\item there exists an automorphic representation $\pi_{U_n}$ of $U_{n}(\mathbf{A})$  such that $\pi_{U_n}$ is the $\xi_{\vec{n}}$-Langlands functorial transfer of  $\pi_{U_{\vec{n}}}$ at all places, and
\item there exists a constant $C \in \mathbf{N}$ depending only upon $\sigma_{U_{\vec{n}}, \lambda}$, such that
		$\dim e_{U_{\vec{n}}}(\pi_{U_{\vec{n}}, f}) \leq C \cdot \dim e_{U_n}(\pi_{U_n, f})$.
\end{itemize}
\end{lemma}
\begin{proof}
Let us  assume the terminology of \cite{white_endoscopy_unitary_group}.
The multiplicity one statement is a special case of \cite[Theorem 11.2]{white_endoscopy_unitary_group}.
Write  $\pi_{U_{\vec{n}}} =  \pi_{U_{\vec{n}}, 1} \times \cdots \times \pi_{U_{\vec{n}}, r}$.
For $i=1,\ldots, r$,
  let
$\Pi_i$ be the Langlands base change of  $\pi_{U_{\vec{n}}, i}$ to $GL_{n_i}(\mathbf{A}_E)$  which
exists as a cuspidal automorphic representation (cf. \cite[Theorem 6.1]{white_endoscopy_unitary_group}).  If $\nu = \upsilon \upsilon'$ is a non-archimedean place of $F$ that splits in $E$, then by the definition of local Langlands base change
\[
	\Pi_{i, \nu} = \Pi_{i, \upsilon} \times \Pi_{i, \upsilon'} \simeq  \pi_{U_{\vec{n}}, i, \nu} \times \pi_{U_{\vec{n}}, i, \nu}^\vee \ \ \forall i =1,\ldots, r.
\]
By a result of Shin \cite[Corollary 1.3]{shin_galois_rep}, the representations $\Pi_{i}$ are tempered at all finite places.
It follows that  $\pi_{U_{\vec{n}}, \nu}$ is tempered. 
To see the existence of $\pi_{U_{{n}}}$, we first define the automorphic representation of $GL_n(\mathbf{A}_E)$
\[
	\Pi' =  \mu_{n - n_1} \Pi_1  \boxplus \cdots \boxplus \mu_{n-n_r} \Pi_r.
\]
Let $\pi_{U_n}$ be an automorphic representation of $U_{n}(\mathbf{A})$ whose Langlands base change is isomorphic to $\Pi'$
whose existence is guaranteed by \cite[Corollary 11.3]{white_endoscopy_unitary_group}.  One  checks from the respective definitions that $\pi_{U_n}$ is the $\xi_{\vec{n}}$-Langlands functorial transfer of  $\pi_{U_{\vec{n}}}$ at all places.

For the final statement, define 
\[
C = \max\left(\left\lceil \frac{{\dim e_{U_{\vec{n}}, \lambda}(\sigma_{U_{\vec{n}}, \lambda})}}{\dim e_{U_{n}, \lambda}(\sigma_{U_n, \lambda})}
\right\rceil : \sigma_{U_n, \lambda} \in \Pi
\right)
\]
where $\Pi$ denotes the $L$-packet of discrete series representations of $U_n(F_\lambda)$ which are the Langlands functorial transfer of $\sigma_{U_{\vec{n}},\lambda}$.  We observe that
$\dim e_{U_{\vec{n}}}(\pi_{U_{\vec{n}}, f}) = \prod_{\nu\in S \cup S_p}\dim e_{U_{\vec{n}}, \nu}(\pi_{U_{\vec{n}}, \nu})$
and $\dim e_{U_{{n}}}(\pi_{U_{{n}}, f}) = \prod_{\nu\in S \cup S_p}\dim e_{U_{{n}}, \nu}(\pi_{U_{{n}}, \nu})$.
If $\nu \in S \cup S_p - \lambda$, then
$\dim e_{U_{\vec{n}}, \nu}(\pi_{U_{\vec{n}}, \nu}) = \dim \pi_{U_{\vec{n}}, \nu}^{I_{U_{\vec{n}, \nu}}}$
(resp. 
$\dim e_{U_{{n}}, \nu}(\pi_{U_{{n}}, \nu}) = \dim \pi_{U_{{n}}, \nu}^{I_{U_{{n}, \nu}}}$)
which is equal to the number of accessible refinements of $\pi_{U_{\vec{n}}, \nu}$ (resp.  $\pi_{U_{{n}}, \nu}$)
(see Lemma \ref{lemma_bernstein_borel_casselman_matsumoto} whose result trivially extends to $\nu \in S \cup S_p - \lambda$).  It follows from  Lemma \ref{lemma_accessible_refinements_mapped_accessible_refinements} (trivially extended to 
$\nu \in S \cup S_p - \lambda$) that
\[
\dim e_{U_{\vec{n}}, \nu}(\pi_{U_{\vec{n}}, \nu}) \leq
\dim e_{U_{{n}}, \nu}(\pi_{U_{{n}}, \nu}).
\]
The result follows.
\end{proof}

\begin{lemma}
\label{lemma_existence_morphism_2}
There exists a $\qpbar$-rigid analytic morphism
$\Xi_2 :  \mathcal{D}' \rightarrow \mathcal{D}_{U_{{n}}}$ for which the
respective diagrams in Section \ref{subsubsection_auxiliary_eigenvarieties} commute.
\end{lemma}
\begin{proof}
We remind the reader that $\Xi_{\mathcal{W}} : \mathcal{W}_{U_{\vec{n}}} \rightarrow \mathcal{W}_{U_n}$ is an isomorphism which allows us to identify the two spaces (cf. Lemma \ref{lemma_weight_space_morphism_isomorphism}).
The existence of such a $\Xi_2$ will follow from the interpolation result of Chenevier \cite[Theorem 1]{chenevier_jacquet_langlands_p_adic} combined with a small slope is classical type result \cite[Proposition 2.17]{chenevier_application_varieties_hecke} upon confirmation that
for all irreducible admissible representations $W_{U_{\vec{n}}}$ of $U_{\vec{n}}(\mathbf{A}_\infty)$ of regular highest weight,  for all irreducible admissible representations $W_{U_n}$ of $U_n(\mathbf{A}_\infty)$ of regular highest weight such that
$\Xi_{\mathcal{W}}(\delta(W_{U_{\vec{n}}})) = \delta_{W_{U_n}}$, and for all $h \in \mathcal{H}^-_{U_n}$, we have that
\[
\det(1 - T h \cdot \phi_{U_{{n}}}|_{S^{\mathrm{cl}}_{U_{\vec{n}}, \kappa(W_{U_{\vec{n}}})} \otimes {\kappa(W_{U_{\vec{n}}})}}) 
|
\det(1 - T h \cdot \phi_{U_{{n}}}|_{S^{\mathrm{cl}}_{U_{{n}}, \kappa(W_{U_{{n}}})} \otimes {\kappa(W_{U_{{n}}})}}) 
\]
as elements of $\mathbf{C}_p[T]$
where we are using the morphism
$\Lambda^* \otimes \xi^*_{\mathrm{ur}} : \mathcal{H}^-_{U_{{n}}} \rightarrow \mathcal{H}^-_{U_{\vec{n}}}$
to view 
$S^{\mathrm{cl}}_{U_{\vec{n}}, \kappa(W_{U_{\vec{n}}})}$ as a $\mathcal{H}^-_{U_n}$-module.
  In fact, the result will follow from a slightly weaker statement namely that
there exists a $C \in \mathbf{N}$ such that 
for all  $W_{U_{\vec{n}}}$, $W_{U_n}$, and $h$  as above, we have that
\[
\det(1 - T h \cdot \phi_{U_{{n}}}|_{S^{\mathrm{cl}}_{U_{\vec{n}}, \kappa(W_{U_{\vec{n}}})} \otimes {\kappa(W_{U_{\vec{n}}})}})
|
\det(1 - T h \cdot \phi_{U_{{n}}}|_{S^{\mathrm{cl}}_{U_{{n}}, \kappa(W_{U_{{n}}})} \otimes {\kappa(W_{U_{{n}}})}})^C.
\]
The fact that this weaker statement suffices follows from the fact that the eigenvariety $\mathcal{D}_{U_n}$ is
seen to be 
 canonically isomorphic to the eigenvariety constructed from the data 
\[
(\mathcal{W}_{U_{{n}}}, \mathbf{S}_{U_{{n}}}^{\oplus C}, \mathcal{H}^-_{U_n}, \phi_{U_n})
\]
where the direct sum of system of (PR) $\qpbar$-Banach modules over $\mathcal{W}_{U_n}$ is defined in the logical way.

The desired statement is a simple  consequence of Lemma \ref{lemma_accessible_refinements_mapped_accessible_refinements}
and Lemma \ref{lemma_classical_transfer_exists_result}, which we shall now explain.
By Lemma \ref{lemma_description_of_module_space_classical_terms_automorphic_representations},
we have the following isomorphisms of $\mathcal{H}^-_{U_n}$-modules, 
\begin{eqnarray*}
S^{\mathrm{cl}}_{U_{{n}}, \kappa(W_{U_{{n}}})} \simeq \bigoplus_{\pi_{U_n}} m(\pi_{U_n}) e_{U_n}(\Pi_{U_n, f}) \otimes_{\mathbf{C}, \iota_p} \qpbar \\
S^{\mathrm{cl}}_{U_{\vec{n}}, \kappa(W_{U_{\vec{n}}})} \simeq \bigoplus_{\pi_{U_{\vec{n}}}} m(\pi_{U_{\vec{n}}}) e_{U_{\vec{n}}}(\Pi_{U_{\vec{n}}, f}) \otimes_{\mathbf{C}, \iota_p} \qpbar \\
\end{eqnarray*}
where $\pi_{U_n}$ (resp. $\pi_{U_{\vec{n}}}$) runs through the automorphic representations of 
$U_n(\mathbf{A})$ (resp. $U_{\vec{n}}(\mathbf{A})$) such that
 $\pi_{U_n, \infty} \simeq W_{U_n}$ (resp. $\pi_{U_{\vec{n}}, \infty} \simeq W_{U_{\vec{n}}}$) and
 $e_{U_n}(\Pi_{U_n, f}) \neq 0$ (resp. $e_{U_{\vec{n}}}(\Pi_{U_{\vec{n}}, f}) \neq 0$), and 
  $\pi_{U_{\vec{n}}}$ is viewed as a $\mathcal{H}^-_{U_n}$-module via the 
  morphism 
  $\Lambda^* \otimes \xi^*_{\mathrm{ur}} : \mathcal{H}^-_{U_{{n}}} \rightarrow \mathcal{H}^-_{U_{\vec{n}}}$.
 By
 Lemma \ref{lemma_bernstein_borel_casselman_matsumoto},
  Lemma \ref{lemma_accessible_refinements_mapped_accessible_refinements}
and Lemma \ref{lemma_classical_transfer_exists_result}, we have an embedding of 
 $\mathcal{H}^-_{U_n}$-modules
 \[
   S^{\mathrm{cl}}_{U_{\vec{n}}, \kappa(W_{U_{\vec{n}}})}
   \otimes {\kappa(W_{U_{\vec{n}}})}
   \hookrightarrow
    \left(S^{\mathrm{cl}}_{U_{{n}}, \kappa(W_{U_{{n}}})} 
 \otimes {\kappa(W_{U_{{n}}})}\right)^{\oplus C}
 \]
 where $C \in \mathbf{N}$ is chosen as in Lemma \ref{lemma_classical_transfer_exists_result}.
The result then follows.

\end{proof}

 \subsubsection{Existence of $\Xi$}
 
 \begin{theorem}
 Let $\sigma_\nu \in \mathfrak{S}_n$ be as in Lemma \ref{lemma_accessible_refinements_mapped_accessible_refinements}
 for all $\nu \in S_p$,  and let the 
 data $(p, S, e_{U_{\vec{n}}}, \phi_{U_{\vec{n}}})$ and $(p, S, e_{U_n}, \phi_{U_n})$ be as in
 Section \ref{subsection_p_adic_functorality_Langlands_direct_sum}.
  Then the $p$-adic Langlands functoriality morphism  
  \[
  	\Xi : \mathcal{D}_{U_{\vec{n}}} \rightarrow \mathcal{D}_{U_n}
  \]
  for the Langlands direct sum exists and is unique.
 \end{theorem}
\begin{proof}
Define $\Xi =  \Xi_2 \circ \Xi_1$ where $\Xi_1$ is the morphism appearing in Lemma \ref{lemma_existence_morphism_1}
and $\Xi_2$ is the morphism appearing in Lemma \ref{lemma_existence_morphism_2}.
The result follows by Lemma \ref{lemma_if_padic_func_exists_it_is_unique}
and Lemma \ref{lemma_langlands_direct_sum_after_fix_sigma_morphisms_weight_unique}.
\end{proof}

\newpage
\bibliographystyle{amsalpha}
\bibliography{padicfunctoriality.bib}

\end{document}